\documentclass[letterpaper, 10 pt, conference]{ieeeconf}  % Comment this line out
\IEEEoverridecommandlockouts                              % This command is only
\usepackage{amsthm}

\usepackage{caption}

\usepackage{graphicx}
\usepackage{color}
\usepackage{algorithm}
\usepackage{algpseudocode}
\usepackage{psfrag}
\usepackage{hyperref}

\usepackage{tikz}
\usetikzlibrary{calc}
\usetikzlibrary{arrows,automata}

\usepackage{amsmath,amssymb}
\usepackage{accents} %for interval bars
\usepackage[font={rm,bf,sc}]{caption}
\usepackage{pifont} % for circled numbers
\usepackage{cite}
\usepackage{booktabs}
\usepackage{multirow}
\usepackage{subcaption}
\usepackage{placeins} % \FloatBarrier
\makeatletter
\let\OldStatex\Statex
\renewcommand{\Statex}[1][3]{%
  \setlength\@tempdima{\algorithmicindent}%
  \OldStatex\hskip\dimexpr#1\@tempdima\relax}
\makeatother

\newtheoremstyle{style}% name of the style to be used
  {\topsep}% measure of space to leave above the theorem. E.g.: 3pt
  {\topsep}% measure of space to leave below the theorem. E.g.: 3pt
  {\itshape}% name of font to use in the body of the theorem
  {0pt}% measure of space to indent
  {\bfseries}% name of head font
  {:}% punctuation between head and body
  { }% space after theorem head; " " = normal interword space
  {\thmname{#1}\thmnumber{ #2}\thmnote{ (#3)}}

\theoremstyle{style}

\newtheorem{definition}{Definition}%alles neu
\newtheorem{proposition}{Proposition}
\newtheorem{theorem}{Theorem}
\newtheorem{lemma}{Lemma}

\newtheorem{assumption}{Assumption}

\newcommand{\R}[1]{\mathbb{R}^{#1}}
\newcommand{\naturals}{\mathbb{N}}
\newcommand{\naturalsof}[1]{\mathbb{N}^{#1}}
\newcommand{\ubar}[1]{\underaccent{\bar}{#1}}
\newcommand{\zeroMat}[1]{0_{#1}}
\newcommand{\id}[1]{I_{#1}}
\newcommand{\norm}[1]{\left \| #1 \right \|_2}

\newcommand{\interval}{\mathcal{I}}
\newcommand{\CH}[1]{\texttt{conv}\big(#1\big)}

\newcommand{\boxOp}[1]{\texttt{box}\big(#1\big)}

\newcommand{\errOp}[1]{\texttt{err}\big(#1\big)}
\newcommand{\incrOp}[1]{\texttt{incr}(#1)}
\newcommand{\decrOp}[1]{\texttt{decr}(#1)}
\newcommand{\resetOp}[1]{\texttt{reset}(#1)}
\newcommand{\redOp}[2]{\texttt{red}\big(#1,#2\big)}
\newcommand{\paramOp}[1]{\texttt{param}\big(#1\big)}
 
\renewcommand{\algref}[1]{Alg.~\ref{#1}} %original in algorithmicx package
\newcommand{\tabref}[1]{Table~\ref{#1}}
\newcommand{\figref}[1]{Fig.~\ref{#1}}
\newcommand{\lineref}[1]{line~\ref{#1}}
\newcommand{\linerangeref}[2]{lines~\ref{#1}-\ref{#2}}
\newcommand{\propref}[1]{Prop.~\ref{#1}}
\newcommand{\defref}[1]{Def.~\ref{#1}}
\newcommand{\thmref}[1]{Theorem~\ref{#1}}
\newcommand{\assref}[1]{Assumption~\ref{#1}}

\newcommand{\lmmref}[1]{Lemma~\ref{#1}}
\newcommand{\sectref}[1]{Sec.~\ref{#1}}

\newcommand\eatpunct[1]{}

\newcommand{\params}{algorithm parameters}
\newcommand{\Paramsheading}{Algorithm Parameters}
\newcommand{\AppModPtextheading}{Propagation Parameters}

\newcommand{\AppModP}{\Phi_{\text{prop}}}
\newcommand{\AppModPidx}[1]{\Phi_{\text{prop,$#1$}}}
\newcommand{\AppModPmax}{\Phi_{\text{prop,max}}}
\newcommand{\SetReprP}{\Phi_{\text{set}}}
\newcommand{\SetReprPidx}[1]{\Phi_{\text{set,$#1$}}}
\newcommand{\etamax}{\eta_{\text{max}}}
\newcommand{\factor}{\mu}
\newcommand{\zonOrder}{\rho}
\newcommand{\bloatmax}{\varepsilon_{\text{max}}}
\newcommand{\bloatHR}{\varepsilon_{\mathcal{H}}}
\newcommand{\bloatHRti}[1]{\varepsilon_{\mathcal{H}}\big(#1\big)}
\newcommand{\bloatHRmax}{\varepsilon_{\mathcal{H},\text{max}}}
\newcommand{\bloatPR}{\varepsilon_{\mathcal{P}}}
\newcommand{\bloatPRti}[1]{\varepsilon_{\mathcal{P}}\big(#1\big)}
\newcommand{\bloatPRtiadm}[1]{\varepsilon_{\mathcal{P},\text{max},#1}}
\newcommand{\bloatPRmax}{\varepsilon_{\mathcal{P},\text{max}}}
\newcommand{\bloatSet}{\varepsilon_{\mathcal{S}}}
\newcommand{\bloatSetti}[1]{\varepsilon_{\mathcal{S}}\big(#1\big)}
\newcommand{\bloatSettiadm}[1]{\varepsilon_{\mathcal{S},\text{max},#1}}
\newcommand{\bloatSetmax}{\varepsilon_{\mathcal{S},\text{max}}}

\newcommand{\uTrans}{c_{u}}
\newcommand{\F}{\mathbf{F}_x}
\newcommand{\Fof}[1]{\mathbf{F}_x(#1)}
\newcommand{\E}{\mathbf{E}}
\newcommand{\Eof}[1]{\mathbf{E}(#1)}

\newcommand{\Eabsof}[1]{E_{\text{abs}}(#1)}

\newcommand{\inputCorr}{\mathbf{F}_u}
\newcommand{\inputCorrof}[1]{\mathbf{F}_u(#1)}

\newcommand{\Gsize}{\gamma}

\newcommand{\inputset}{\mathcal{U}}
\newcommand{\initset}{\mathcal{X}^0}

\newcommand{\HRset}{\mathcal{H}^\mathcal{R}}
\newcommand{\PRset}{\mathcal{P}^\mathcal{R}}

\newcommand{\HReqset}{\mathcal{H}^\mathcal{R}_=}

\newcommand{\HRplusset}{\mathcal{H}^\mathcal{R}_+}

\newcommand{\HRti}[1]{\mathcal{H}^\mathcal{R}\big(#1\big)}
\newcommand{\PRti}[1]{\mathcal{P}^\mathcal{R}\big(#1\big)}
\newcommand{\PRtiunderbrace}[1]{\mathcal{P}^\mathcal{R}{\textstyle(}#1{\textstyle)}}
\newcommand{\Rti}[1]{\mathcal{R}\big(#1\big)}
\newcommand{\HRexti}[1]{\mathcal{H}^\mathcal{R}_{\text{ex}}\big(#1\big)}
\newcommand{\PRexti}[1]{\mathcal{P}^\mathcal{R}_{\text{ex}}\big(#1\big)}
\newcommand{\Rexti}[1]{\mathcal{R}_{\text{ex}}\big(#1\big)}

\newcommand{\HReqti}[1]{\mathcal{H}^\mathcal{R}_=\big(#1\big)}
\newcommand{\PReqti}[1]{\mathcal{P}^\mathcal{R}_=\big(#1\big)}
\newcommand{\HRplusti}[1]{\mathcal{H}^\mathcal{R}_+\big(#1\big)}
\newcommand{\PRplusti}[1]{\mathcal{P}^\mathcal{R}_+\big(#1\big)}
\newcommand{\HReqtiunderbrace}[1]{\mathcal{H}^\mathcal{R}_={\textstyle(}#1{\textstyle)}}
\newcommand{\PReqtiunderbrace}[1]{\mathcal{P}^\mathcal{R}_={\textstyle(}#1{\textstyle)}}
\newcommand{\HRplustiunderbrace}[1]{\mathcal{H}^\mathcal{R}_+{\textstyle(}#1{\textstyle)}}
\newcommand{\PRplustiunderbrace}[1]{\mathcal{P}^\mathcal{R}_+{\textstyle(}#1{\textstyle)}}

\newcommand{\Sarb}{\mathcal{S}}
\newcommand{\Splus}{\mathcal{S}_{+}}
\newcommand{\Seq}{{\mathcal{S}_{=}}}
\newcommand{\Stot}{{\mathcal{S}_{\text{tot}}}}

\title{\LARGE \bf
Preparation of Papers for IEEE CSS Sponsored Conferences \& Symposia
}

\usepackage{tikz} % import tikz package if you haven't done it yet
\newcommand\copyrighttext{%
	\footnotesize \copyright 2020 IEEE. Personal use of this material is permitted. Permission from IEEE must be obtained for all other uses, in any current or future media, including reprinting/republishing this material for advertising or promotional purposes, creating new collective works, for resale or redistribution to servers or lists, or reuse of any copyrighted component of this work in other works.}
\newcommand\copyrightnotice{%
	\begin{tikzpicture}[remember picture,overlay]
		\node[anchor=south,yshift=10pt] at (current page.south) {\fbox{\parbox{\dimexpr\textwidth-\fboxsep-\fboxrule\relax}{\copyrighttext}}};
	\end{tikzpicture}%
}

\begin{document}

%\title{Adaptive Parameter Tuning for Reachability Analysis of Linear Systems}

%\numberofauthors{3}
%\author{
	%% 1st. author
	%\alignauthor
	%Mark Wetzlinger\\
	%\affaddr{Technische Universit\"at M\"unchen}\\
	%\email{m.wetzlinger@tum.de}
	%% 2nd. author
	%\alignauthor
	%Niklas Kochdumper\\
	%\affaddr{Technische Universit\"at M\"unchen}\\
	%\email{niklas.kochdumper@tum.de}
	%% 3rd. author
	%\alignauthor
	%Matthias Althoff\\
	%\affaddr{Technische Universit\"at M\"unchen}\\
	%\email{althoff@tum.de}
%}

%\maketitle

%\title{Adaptive Parameter Tuning for Reachability Analysis of Linear Systems\thanks{Supported by organization x.}}
\title{\LARGE \bf Adaptive Parameter Tuning for Reachability Analysis of Linear Systems}

\author{Mark Wetzlinger, Niklas Kochdumper, and Matthias Althoff% <-this % stops a space
%\thanks{This work was not supported by any organization}% <-this % stops a space
\thanks{M. Wetzlinger, N. Kochdumper, and M. Althoff are with the Faculty of Computer Science,
        Technical University of Munich, Garching, Germany
        \{{\tt\small m.wetzlinger@tum.de}, {\tt\small niklas.kochdumper@tum.de}, {\tt\small althoff@tum.de}\}.
        The authors gratefully acknowledge partial financial supports
		from the research training group CONVEY funded by
		the German Research Foundation under grant GRK 2428.}%
}
%\author{Mark Wetzlinger\inst{1}\orcidID{0000-1111-2222-3333} \and
%Niklas Kochdumper\inst{1}\orcidID{1111-2222-3333-4444} \and
%Matthias Althoff\inst{1}\orcidID{2222--3333-4444-5555}}
%
%\authorrunning{M. Wetzlinger et al.}
%\institute{Technische Universität München, Boltzmannstraße 3, 85748 Garching bei München, Germany
%\email{\{m.wetzlinger,niklas.kochdumper,althoff\}@tum.de}}
%
\maketitle              % typeset the header of the contribution

\begin{abstract}
	
Despite the possibility to quickly compute reachable sets
of large-scale linear systems,
current methods are not yet widely applied by practitioners.
The main reason for this is probably that current approaches
are not push-button-capable and still require to manually set crucial parameters,
such as time step sizes and the accuracy of the used set representation---these
settings require expert knowledge.
We present a generic framework
to automatically find near-optimal parameters for reachability analysis
of linear systems given a user-defined accuracy.
To limit the computational overhead as much as possible,
our methods tune all relevant parameters during runtime.
We evaluate our approach on benchmarks from the ARCH competition as well as on random examples.
The results show that our new framework verifies the selected benchmarks
faster than manually-tuned parameters and is an order of magnitude
faster compared to genetic algorithms.

%Reachability analysis of linear continuous systems has made tremendous progress over the last years. However, despite the possibility to quickly compute reachable sets of large-scale systems for verification purposes, current methods are not yet widely applied by practitioners. We believe that the main reason for this is that current approaches are not push-button-capable and still require to manually set crucial parameters, such as time step sizes and the accuracy of the used set representation---these settings require expert knowledge. In order to realize an algorithm that is fully automatic, we present for the first time a generic framework to automatically find near-optimal parameters for reachability analysis of linear systems given a user-defined accuracy. To limit the computational overhead as much as possible, our methods tune all relevant parameters during runtime. We evaluate our approach on benchmarks from the ARCH competition as well as on random examples. Our results show that our new framework verifies the benchmarks faster than manually-tuned parameters and is orders of magnitude faster compared to genetic algorithms.
 	
\end{abstract}

%\category{\question{I.2.8}}{\question{Artificial Intelligence}}{\question{Problem Solving, Control Methods, and Search}}[\question{Control theory}]
%\category{}{}{}[]

%\terms{}

%\keywords{\begin{sloppypar}
%		Linear time-invariant systems, reachability analysis, adaptive parameter tuning.
%\end{sloppypar}}

%\keywords{Linear time-invariant systems, reachability analysis, adaptive \linebreak parameter tuning.}
\copyrightnotice

% -------------------------------------------------------
% INTRODUCTION

\section{Introduction} 
\label{sec:Introduction}

% use other papers on similar topics for inspiration
% main parts: 1. motivation, 2. state-of-the-art, 3. contribution

%\todo{Motivation}

Reachability analysis is one of the main techniques
to formally verify the correctness of mixed discrete/continuous systems:
It computes the set of reachable states of a system
for a set of uncertain initial states as well as uncertain inputs.
If the reachable set does not intersect the unsafe regions
defined by a given safety property, that property is satisfied.

Exact reachable sets can only be computed
for a limited class of systems \cite{LafPap2001}.
In the remaining cases, one calculates over-approximations
whose tightness is strongly influenced by \params{};
unlike model parameters, these parameters have no relation to the considered model.
%However, these \params{} have a major role the computation process,
%such as the time step size with which reachable sets are propagated.

Due to large over-approximations resulting from poor \params{},
reachability analysis may fail to prove a safety property
even though the property is satisfied by the exact reachable set.
We address this problem by proposing a novel generic framework
to automatically tune all \params{} for reachability analysis
of linear continuous time-invariant systems during runtime
respecting a user-defined error bound.
Our framework can be used for different reachability algorithms
and for different set representations.

\paragraph{State of the Art}

Reachability algorithms for linear continuous systems
are mainly based on the propagation of reachable sets
for multiple time steps until reaching a fixed point or a user-defined time horizon.
Propagation-based techniques have been extensively investigated
\cite{Althoff2007c,Che2015,Gurung2019,Girard2006,BFFPSV2018,Frehse2011}
using tools such as
\textit{CORA} \cite{Althoff2015a}, \textit{Flow*} \cite{Che2013},
\textit{XSpeed} \cite{Ray2015b}, \textit{SpaceEx} \cite{Frehse2011},
and \textit{JuliaReach} \cite{BogFor2019}.
%One group of propagation-based techniques uses the solution from the previous time step
%\cite{Althoff2007c,Che2015,Gurung2019,Girard2006},
%which are, e.g., implemented in the tools \textit{CORA} \cite{Althoff2015a},
%\textit{Flow*} \cite{Che2013}, and \textit{XSpeed} \cite{Ray2015b}.
%Another group computes reachable sets from the initial set of states
%\cite{BFFPSV2018,FreLeG2011}, which are, e.g., implemented
%in the tools \textit{JuliaReach} \cite{BogFor2019}
%and \textit{SpaceEx} \cite{FreLeG2011}.
A further method is to compute reachable sets using simulations
\cite{Bak2017b,Duggirala2016}, which is implemented
in the tool \textit{HyLAA} \cite{Bak2017}.
The used set representation is another distinctive feature
besides the method to compute reachable sets.
Many different set representations have been researched,
including polytopes \cite{Frehse2011},
zonotopes \cite{Girard2006}, ellipsoids \cite{Kurzhanski2000},
griddy polyhedra \cite{AsaDan2000}, star sets \cite{Bak2017b},
support functions \cite{Girard2008b},
and constrained zonotopes \cite{Scott2016}.

Previous approaches for automated parameter tuning
mainly focused on the time step size:
Numerical ODE solvers often compute different solutions in parallel
and decrease the time step size if the difference
between solutions exceeds a certain threshold
\cite{Lapidus1971,Ascher1994}.
Furthermore, several automated time step adaptation strategies
have been developed for \textit{guaranteed integration} methods
that enclose only a single trajectory rather than a set of trajectories
\cite{Kerbl1991,Rufeger1993,Nedialkov2000}.
For reachability analysis, only a few methods automatically set the time step size.
The approach in \cite{Frehse2011} chooses the time step size
automatically in order to keep the error in a user-defined direction
below a user-defined bound.
Furthermore, the approach in \cite{PraVis2011} automatically chooses
suitable time step sizes for affine systems
so that the Hausdorff distance between the reachable set
obtained by the exact flow and the approximated flow stays below a certain threshold.
So far, there is no approach that considers all \params{}.

\paragraph{Contributions}

We introduce a novel generic framework that automatically tunes
all \params{} of a reachability algorithm
so that the over-approximation error stays below a user-defined threshold.
This generalized framework can be applied to
many different reachability algorithms
and we show that our self-parametrization approach always converges
if the reachability algorithm satisfies certain requirements.

\medskip
After introducing some preliminaries in \sectref{sec:prelim},
we present our generic framework in \sectref{sec:selfparam}.
In \sectref{sec:implementation}, we demonstrate the implementation
of our framework for a specific reachability algorithm.
Finally, the evaluation of our implementation on several numerical examples
in \sectref{sec:numex} demonstrates
%that our framework enables
the fully automated computation
of tight over-approximations of reachable sets
with only little computational overhead. 

% -------------------------------------------------------
% RA for LINSYS

\section{Preliminaries}
\label{sec:prelim}

To concisely explain the novelties of our paper,
we first introduce some preliminaries.

\subsection{Notation}
\label{ssec:notation}

% NOTATION
% The set of real numbers is denoted by~$\reals{}$.
% 0, id
Vectors are denoted by lower-case letters
and matrices by upper-case letters.
Square matrices of zeros are denoted by
${\zeroMat{n} \in \R{n \times n}}$
and the identity matrix by ${\id{n} \in \R{n \times n}}$.
Given a vector ${a \in \R{n}}$, $a_i$ refers to the $i$-th entry.
% intervals
An  $n$-dimensional interval is denoted by
$\interval{} = [\ubar{l},\bar{l}] \subset \R{n}$,
where $\ubar{l}_i \leq \bar{l}_i$, $\forall i \in \{1,...,n\}$.
The operators $\inf(\interval{}) = \ubar{l} \in \R{n}$
and $\sup(\interval{}) = \bar{l} \in \R{n}$
return the infimum and the supremum of
an interval $\interval{} = [\ubar{l},\bar{l}]$.
% interval matrices
Interval matrices are denoted in boldface:
$\mathbf{I}=[\ubar{I},\bar{I}] \in \R{m \times n}$,
where $\inf(\mathbf{I}) = \ubar{I} \in \R{m \times n}$
and $\sup(\mathbf{I}) = \bar{I} \in \R{m \times n}$.
%Symmetric interval matrices are matrices where
%$\sup(\mathbf{I}) = -\inf(\mathbf{I})$.
% set operations
%The enclosure of any set by an interval
%is denoted by the operator $\boxOp{\cdot}$,
%the absolute value by
%$\absOp{\interval{}} := \max\{ |\inf(\interval{})|, |\sup(\interval{})| \}$,
%where the max operator is applied elementwise.
The linear map is written without any operator between the operands,
the Minkowski sum is denoted by $\oplus$,
and the convex hull operator by $\CH{\cdot}$.
%The concatenation of two matrices $C \in \R{n \times c}$
%and $D \in \R{n \times d}$ is denoted by $[C,D]$.
% functions
%Finally, the function $\norm{\cdot}$ denotes the Euclidean norm.
%Finally, the ceil operator $\ceilfun{a}$ rounds every $a \in \reals$
%to the next higher integer value.

\subsection{Reachability Analysis of Linear Systems}
\label{ssec:ra_linsys}

The presented technique for adaptive self-parametrization
is applied to linear time-invariant (LTI) systems
%
%\begin{definition} \label{def:LTI} (Linear time-invariant systems)
%An LTI system
with initial states bounded by $\initset{} \subset \R{n}$
and inputs bounded by $\inputset{} \subset \R{n}$:
%is defined as
\begin{align} 
	\begin{split}  \label{eq:LTI}
	&\dot{x}(t) = A x(t) + u(t) \; , \\
	&\text{with} \quad x(0) \in \initset{} \subset \R{n}, \;
		\forall t: u(t) \in \inputset{} \subset \R{n} \; ,
	\end{split}
\end{align}
where $A \in \R{n \times n}$ is the system matrix,
$x(t) \in \R{n}$ is the state vector,
and $u(t) \in~\R{n}$ is the input vector.
%\hfill $\square$
%\end{definition}

The above system also encompasses systems
$\dot{x}(t) = A x(t) + B \tilde{u}(t)$,
where $B \in \R{m \times n}$ is the input matrix,
as we can set $\inputset = \{B\tilde{u} \, | \, \tilde{u} \in \mathcal{D} \subset \R{m}\}$.
Let us introduce $\xi \big(t; x_0, u(\cdot) \big)$
as the solution of \eqref{eq:LTI} to define the exact reachable set
%
%\begin{definition} \label{def:Rexset} (Reachable sets)
%The exact reachable set
$\Rexti{[0,t_f]}$ of \eqref{eq:LTI} over the time horizon $t \in [0,t_f]$:
\begin{align*}
%	\hspace{7pt}
	\Rexti{[0,t_f]} &= \Big\{ \xi \big(t; x_0, u(\cdot) \big) \, \Big| \, x_0 \in \initset{}, \\
	&\qquad \quad \forall \tau \in [0,t]: u(\tau) \in \inputset{}, \, t \in [0,t_f] \Big\} \; .
%	 \hspace{6pt} \square
\end{align*}
%\hfill $\square$
%\end{definition}

Due to the superposition principle, $\Rexti{[0,t_f]}$ is
the sum of the homogeneous solution~$\HRexti{[0,t_f]}$
and the inhomogeneous solution~$\PRexti{[0,t_f]}$ \cite[(3.6)]{Althoff2010a}:
\begin{equation*}
	\Rexti{[0,t_f]} = \HRexti{[0,t_f]} \oplus \PRexti{[0,t_f]} \; ,
\end{equation*}
where
\begin{align*}
	\HRexti{[0,t_f]} &= \Big\{ e^{A t} x_0 \, \Big|
		\, x_0 \in \initset{}, \, t \in [0,t_f] \Big\} \; , \\ %\label{eq:HRexti} \\
%	\PRexti{[0,t_f]} &=
	\PRexti{t_f} &= \bigg\{ \int_0^{t_f} e^{A (t_f-\tau)} u(\tau) \, \mathrm{d}\tau \, \bigg|
			\, u(\tau) \in \inputset{} \bigg\} \; . %\label{eq:PRexti}
\end{align*}
%\begin{align*}
%	\Rexti{[0,T]} &= \HRexti{[0,t_f]} \oplus \PRexti{[0,t_f]} \; , \\
%	\text{where} \; \HRexti{[0,t_f]} &= \Big\{ e^{A t} x_0 \, \Big|
%		\, x_0 \in \initset{}, \, t \in [0,t_f] \Big\} \; , \\ %\label{eq:HRexti} \\
%%	\PRexti{[0,T]} &=
%	\PRexti{T} &= \bigg\{ \int_0^{t_f} e^{A (T-\tau)} u(\tau) \mathrm{d}\tau \, \bigg|
%			\, u(\tau) \in \inputset{} \bigg\} \; . %\label{eq:PRexti}
%\end{align*}
%
%The equality $\PRexti{[0,t_f]} = \PRexti{t_f}$
%holds only if $0 \in \inputset{}$,
%which is the case we consider in this work.
%The extension for $0 \notin \inputset{}$
%is shown in \cite[Sec. 3.2.2]{Althoff2010a}.
If $0 \in \inputset{}$, $\PRexti{[0,t_f]} = \PRexti{t_f}$
and the extension for $0 \notin \inputset{}$
is shown in \cite[Sec. 3.2.2]{Althoff2010a}.
%As mentioned in the introduction, reachable sets cannot be computed exactly
%\cite{LafPap2001}
%so that we compute tight over-approximations
%$\HRexti{[0,t_f]} \subseteq \HRti{[0,t_f]}$
%and~$\PRexti{[0,t_f]} \subseteq \PRti{[0,t_f]}$.
%
To limit the over-approximation,
the time horizon $[0,t_f]$ is discretized by $K$ time steps
$\Delta t_i = t_{i+1} - t_i > 0$, $i \in \{0,...,K-1\}$,
where $t_f = \sum_{i=0}^{K-1} \Delta t_i$
so that
%The reachable set of the whole time horizon is obtained as
$\bigcup_{i=0}^{K-1}~\Rti{[t_i,t_{i+1}]}$.
%There are two possible ways of computing the reachable sets
%$\HRset{}$ and $\PRset{}$ of the $i$-th time step:
%\begin{enumerate}
%	\item %from former
%		\textcolor{red}{
%		Reachable sets can be propagated using the previous solution.
%		For constant time step sizes ($\forall i: \Delta t_i = \Delta t$),
%		one obtains \cite[Sec. 2]{Girard2006}:
		%
%		\begin{align*}
%			\HRti{[t_i,t_{i+1}]} &= e^{A \Delta t} \HRti{[t_{i-1},t_i]} \; , \\
%			\PRti{[0,t_{i+1}]} &= e^{A \Delta t} \PRti{[0,t_i]} \oplus \PRti{[0,\Delta t]} \; .
%		\end{align*}
%		\begin{align}
%			\begin{split} \label{eq:HRPR_former}
%				\HRti{[t_i,t_{i+1}]} &= e^{A \Delta t} \HRti{[t_{i-1},t_i]} \; , \\
%				\PRti{[0,t_{i+1}]} &= e^{A \Delta t} \PRti{[0,t_i]} \oplus \PRti{[0,\Delta t]} \; .
%			\end{split}
%		\end{align}
		%
%		Since the inhomogeneous solution accumulates over time for $0 \in \inputset{}$,
%		we have that $\PRti{[t_i,t_{i+1}]} = \PRti{[0,t_{i+1}]}$.
%		For constant time step sizes, the exponential matrix $e^{A \Delta t}$
%		does not have to be re-computed,
%		while for varying time step sizes $\Delta t_i \neq \Delta t_{i-1}$,
%		one additionally requires the reachable sets $\HRtp{t_i}$ and $\PRti{t_i}$
%		as the new initial sets for the remaining time horizon $[t_i,T]$.
%		This method is presented in
%		\cite{Althoff2007c,Che2015,Gurung2019,Girard2006}
%		and implemented in the tools \textit{CORA} \cite{Althoff2015a},
%		\textit{Flow*} \cite{Che2013}, and \textit{XSpeed} \cite{Ray2015b}.
%		}
%	\item %from start
		We compute the reachable sets from the first time interval
		\cite[Chap. 4]{LeG2009}:
		\begin{align}
			\begin{split} \label{eq:HRPR_start}
				\HRti{[t_i,t_{i+1}]} &= e^{A t_i} \HRti{[0,\Delta t_i]} \; ,  \\
				\PRti{[0,t_{i+1}]} &= \PRti{[0, t_i]} \oplus
					\underbrace{e^{A t_i} \PRti{[0,\Delta t_i]}}_%
					{=: \, \PRtiunderbrace{[t_i,t_{i+1}]}} \; .
			\end{split}
		\end{align}
		This method adapts seamlessly to varying time step sizes,
		making it the preferred choice for adaptive parameter tuning
		compared to methods propagating the reachable set from the previous time step.
%		It is presented in \cite{BFFPSV2018,LeG2009}
%		and implemented in the tools \textit{SpaceEx} \cite{FreLeG2011}
%		and \textit{JuliaReach} \cite{BogFor2019}.
%\end{enumerate}

%\begin{enumerate}
	%\item Homogeneous time interval solutions need to cover
		%all trajectories $x(t)$ starting at $t_i$ and ending at $t_{i+1}$.
		%In order to account for the curvature of those trajectories,
		%the solution is overapproximated.
	%\item Generally, the input $u(t)$ is not constant over $[t_i, t_{i+1}]$,
		%making an overapproximation necessary to account for all possible
		%$u(t) \in \inputset{}$ within the time interval. 
	%\item The computation of the exponential matrix $e^{A t_i}$ propagating
		%the homogeneous and inhomogeneous solutions forward in time is only 
		%computed up to a certain precision.
		%The remainder is overapproximated by a residual term
		%affecting the propagation of every set.
%\end{enumerate}

% ------------------------------------------------------------------------------------
% ------------------------------------------------------------------------------------

% -------------------------------------------------------
% SELF-PARAMETRIZATION

\section{Self-Parametrization}
\label{sec:selfparam}

In this section, we introduce our novel algorithm for adaptive parameter tuning.

\subsection{Overview of \Paramsheading{}}
\label{ssec:paramoverview}

The reachability algorithm from \sectref{ssec:ra_linsys}
%using the method in~\eqref{eq:HRPR_start}
depends on different \params{},
which we divide into:
\begin{itemize}
	\item \textbf{Time Step Size:}
		The propagation in~\eqref{eq:HRPR_start} can be applied to any series
		of $\Delta t_i$ adding up to $t_f$.
		Large $\Delta t$ speed up the computation, while small $\Delta t$ increase the precision.
%		For computational efficiency, $\Delta t$ is often kept constant to avoid re-computing
%		$\HRti{[0,\Delta t]}$ and $\PRti{[0,\Delta t]}$.
	\item \textbf{\AppModPtextheading{}:}
		The tightness of the computed reachable sets
		also depends on parameters approximating the dynamics $\AppModP{}$,
		e.g., determining the precision of computing $e^{At}$.
%		$\HRset{}$ and $\PRset{}$ depend on the time step size $\Delta t$
%		and additional \appModPtext{} $\AppModP{}$.		
%		The tightness of the solution depends monotonically
%		on the valuation of the parameters $\AppModP{}$:
		We introduce the operator $\incrOp{\AppModP{}}$
		which increments the parameters $\AppModP{}$ towards values
		which result in a tighter reachable set,
		as well as the operator $\resetOp{\AppModP{}}$
		which resets~$\AppModP{}$ to the coarsest setting.
	\item \textbf{Set Representation:}
		We denote the number of scalar values needed to describe a set by~$\SetReprP{}$.
		Let us introduce the operator $\decrOp{\SetReprP{}}$ which decreases $\SetReprP{}$.
		If $\SetReprP{}$ has reached its minimum, it returns $\decrOp{\SetReprP{}} = \SetReprP{}$.
		We denote the over-approximation of a set $\Sarb{}$
		by reducing its number of parameters
		to a desired value $\SetReprP{}$ by the operator $\redOp{\Sarb{}}{\SetReprP{}}$.
\end{itemize}
%
%Our goal is to extend the reachability algorithm presented in \sectref{ssec:ra_linsys}
%so that the \params{} are tuned automatically.

% ------------------------------------------------------------------------------------
% ------------------------------------------------------------------------------------

\subsection{Error measures}
\label{ssec:errors}

As we can only compute over-approximations, % (see \sectref{ssec:ra_linsys}),
we are interested in measuring the over-approximation error.
Let the reachable set $\HRset{}$ consist of summands, e.g.,
$\HRset{} = \HRset{}^{(1)} \oplus ... \oplus \HRset{}^{(\nu)}$,
as shown for a concrete implementation in \sectref{ssec:setprop}.
This assumption is valid for almost all approaches, see e.g.
\cite{Althoff2010a,Girard2005,Girard2008b,Frehse2011,BFFPSV2018}.
We sum over all indices $\mathcal{E}$ representing exactly-computed terms
to obtain $\HReqset{} = \bigoplus_{i \in \mathcal{E}}^{\nu} \HRset{}^{(i)}$
and the remaining terms are summed to $\HRplusset{} = \bigoplus_{i \notin \mathcal{E}}^{\nu} \HRset{}^{(i)}$.
The same is done for $\PRset{}$, so that we obtain
\begin{align}
	\HRti{[0,\Delta t]} &= \HReqti{[0,\Delta t]} \oplus \HRplusti{[0,\Delta t]} \; ,
		\label{eq:HRtieqplus} \\
	\PRti{[0,\Delta t]} &= \PReqti{[0,\Delta t]} \oplus \PRplusti{[0,\Delta t]} \; .
		\label{eq:PRtieqplus}
\end{align}
We use this separation to realize a computationally efficient over-approximation
of the Hausdorff distance
%$d_H$, which for two sets $\Sarb{}_1, \Sarb{}_2 \subset \R{n}$ is given by
%
\begin{align*}
%	\begin{split} \label{eq:dH}
	d_H (\Sarb{}_1, \Sarb{}_2)
		= \max &\Big\{ \sup_{x_1 \in \Sarb{}_1}
				\big( \, \inf_{x_2 \in \Sarb{}_2} \, \norm{x_1 - x_2} \,
			\big), \\
		&\quad \sup_{x_2 \in \Sarb{}_2}
			\big( \, \inf_{x_1 \in \Sarb{}_1} \norm{x_1 - x_2} \,
			\big) \Big\} \; .
%	\end{split}
\end{align*}
\begin{proposition} \label{prop:bloat} % (Error due to Minkowski sum)
Let $\Seq{} \subset \R{n}$ and $\Splus{} \subset \R{n}$
with ${0 \in \Splus{}}$ be non-empty compact, convex sets.
The Hausdorff distance between $\Seq{}$ and $\Stot{} := \Seq{} \oplus \Splus{}$
%(see \figref{subfig:SeqSplus})
can be over-approximated by
\begin{equation} \label{eq:bloat}
	d_H(\Seq{},\Stot{}) \leq \errOp{\Splus{}}
		:= r \big( \boxOp{\Splus{}} \big) \; ,
%		:= \norm{ \sup \big( \absOp{ \boxOp{ \Splus{} } } \big) } \; ,
\end{equation}
where $\boxOp{\Splus{}}$ is the box enclosure of $\Splus{}$
and $r(\cdot)$ returns the radius of the smallest hypersphere
centered at the origin enclosing its argument.
\end{proposition}
\begin{proof} 
Since $\Stot{}$ encloses $\Seq{}$, we have
\begin{equation} \label{eq:dH_r}
	d_H(\Seq{},\Stot{})
		= \sup_{y \in \Stot{}} \big( \, \inf_{x \in \Seq{}} \, \norm{x - y} \big) \; .
\end{equation}
We also know that the difference between $\Stot{}$ and $\Seq{}$ is given by $\Splus{}$,
so that \eqref{eq:dH_r} can be equivalently written as
\begin{equation*}
	\sup_{y \in \Stot{}} \big( \, \inf_{x \in \Seq{}} \norm{x - y} \, \big)
		\overset{0 \in \Splus{}}{=} \sup_{x \in \Splus{}} \norm{x} = r ( \Splus{} ) \; .
\end{equation*}
Obviously, $r(\Splus{}) \leq r\big(\boxOp{\Splus{}}\big)$
and therefore $d_H(\Seq{},\Stot{}) \leq \errOp{\Splus{}}$.
\end{proof}

%The error measure \eqref{eq:bloat} is subsequently used to bound the errors
%caused by the over-approximation in the reachability analysis.
Let us introduce user-defined upper bounds for errors:
$\bloatHRmax{}$ for the error of the homogeneous solution $\HRset{}$ and
$\bloatPRmax{}$ for the error of the inhomogeneous solution $\PRset{}$.
Our approach ensures that these errors are not exceeded.
%in the calculated reachable set do not exceed these bounds.

We now compute the errors of the $i$-th time step.
For the error $\bloatHRti{[t_i,t_{i+1}]}$ and $\bloatPRti{[t_i,t_{i+1}]}$
we use the split in~\eqref{eq:HRtieqplus} and apply the error measure from \propref{prop:bloat}:
\begin{align}
	\bloatHRti{[t_i,t_{i+1}]} &:= \errOp{\HRplusti{[t_i,t_{i+1}]}} \label{eq:bloatHRti} \\
		&\geq d_H \big( \HReqti{[t_i,t_{i+1}]},\HRti{[t_i,t_{i+1}]} \big) \; , \nonumber \\
	\bloatPRti{[t_i,t_{i+1}]} &:= \errOp{\PRplusti{[t_i,t_{i+1}]}} \label{eq:bloatPRti} \\
		&\geq d_H \big( \PReqti{[t_i,t_{i+1}]},\PRti{[t_i,t_{i+1}]} \big) \; . \nonumber
\end{align}
%
%Using the method in~\eqref{eq:HRPR_start},
From \eqref{eq:HRPR_start}, we see that the computation of $\bloatHR{}$ does not depend on previous time intervals.
Contrary, the inhomogeneous solution $\PRti{[0,t_i]}$
accumulates over time, see \eqref{eq:HRPR_start},
and with it the error $\bloatPR{}$:
\begin{equation} 
	\bloatPRti{[0,t_{i+1}]} = \bloatPRti{[0,t_i]} + \bloatPRti{[t_i,t_{i+1}]} \; ,
		\label{eq:bloatPRti_acc}
\end{equation}
where $\bloatPRti{[0,0]} = 0$.
Let us introduce a mild assumption for the errors $\bloatHR{}$ and $\bloatPR{}$,
which will be justified in \sectref{ssec:verass}.

\begin{assumption} \label{ass:bloatHRPR} (Convergence of $\bloatHR{}$ and $\bloatPR{}$)
We neglect floating-point errors so that $\bloatHRti{[t_i,t_{i+1}]} \rightarrow 0$
and $\bloatPRti{[t_i,t_{i+1}]} \rightarrow 0$ for $\Delta t_i \rightarrow 0$. \hfill $\square$
\end{assumption}
\assref{ass:bloatHRPR} is mild, since it practically holds
for all other current approaches, see e.g. \cite{Althoff2015a, Frehse2011, BogFor2019}.
For further derivations we need the following definition.
\begin{definition} \label{def:superlinear} (Superlinear decrease)
	The set-based evaluation $f(\Sarb{}(t)) = \{ f(x) \, | \, x \in \Sarb{}(t) \}$
	of a continuous function $f: \R{n} \rightarrow \R{}$
	decreases superlinearly if
	\begin{equation*} %\label{eq:superlinear}
		\hspace{8pt} \forall \varphi \in (0,1):
			f\big(\Sarb{}\big([t,t + \varphi \Delta t]\big)\big)
				\subseteq \varphi \, f\big(\Sarb{}\big([t,t + \Delta t]\big)\big) \; .
		\hspace{5pt} \square
	\end{equation*}
%	\hfill $\square$
\end{definition}
%
%We cover this demand by bounding $\bloatPRti{[t_i,t_{i+1}]}$
%by an admissible error $\bloatPRtiadm{i}$.
The following proposition addresses satisfying $\bloatPRmax{}$.

\begin{proposition} \label{prop:bloatPRtiadm}
Let the error $\bloatPRti{[t_i,t_{i+1}]}$
decrease superlinearly according to \defref{def:superlinear}.
For any given $\bloatPRmax{} > 0$, there exists a sequence of
time intervals $[t_i,t_{i+1}]$ so that
\begin{equation*} %\label{eq:bloatPRsum}
	\sum_i \bloatPRti{[t_i,t_{i+1}]} \leq \bloatPRmax{} \; .
\end{equation*}
\end{proposition}

\noindent \textit{Proof.}
%\begin{proof}
We define an admissible error $\bloatPRtiadm{i}$ for each step:
\begin{equation} \label{eq:bloatPRtiadm}
%	\forall i \in \{ 0, ..., K-1 \} :
	\bloatPRti{[t_i,t_{i+1}]} \leq
	\bloatPRtiadm{i} := \frac{\bloatPRmax{} - \bloatPRti{[0,t_i]}}{t_f - t_i} \, \Delta t_i .
\end{equation}
For $\Delta t_i \rightarrow 0$, $\bloatPRtiadm{i}$ converges to 0,
as does $\bloatPRti{[t_i,t_{i+1}]}$ by \assref{ass:bloatHRPR}.
Moreover, $\bloatPRtiadm{i}$ decreases linearly in~$\Delta t_i$,
whereas $\bloatPRti{[t_i,t_{i+1}]}$ decreases superlinearly by assumption.
Hence, we can always obtain a $\Delta t_i$ so that the inequality in \eqref{eq:bloatPRtiadm} holds.
%Thus, \eqref{eq:bloatPRtiadm} holds for all $i$, since the error curve
%of $\bloatPRti{[t_i,t_i + \Delta t_i]}$ for increasing $\Delta t_i$
%starting at $\Delta t_i = 0$ is below the linear slope of $\bloatPRtiadm{i}$.
We sum each side in \eqref{eq:bloatPRtiadm} to
\begin{equation} \label{eq:bloatPRsumti}
	\bloatPRti{[0,t_i]} \leq \sum_{j=0}^{i-1} \bloatPRtiadm{j}
\end{equation}
which is subsequently used to bound $\bloatPRtiadm{i}$ and $\bloatPRmax{}$: % proof \eqref{eq:bloatPRsum}:
\begin{align*}
	\hspace{26pt}
	\bloatPRtiadm{i}
		&\overset{\eqref{eq:bloatPRtiadm}}{=}
			\big( \bloatPRmax{} - \bloatPRti{[0,t_i]} \big) \frac{\Delta t_i}{t_f - t_i} \\
		&\overset{\eqref{eq:bloatPRsumti}}{\leq}
			\big( \bloatPRmax{} - \sum_{j=0}^{i-1} \bloatPRtiadm{j} \big) \frac{\Delta t_i}{t_f - t_i} \\
	\Rightarrow \bloatPRmax{}
		&\geq \sum_{j=0}^{i-1} \bloatPRtiadm{j}
			+ \underbrace{\frac{t_f - t_i}{\Delta t_i}}_{\geq 1} \bloatPRtiadm{i} \\
		&\geq \sum_{j=0}^{i} \bloatPRtiadm{j} \geq \sum_{i} \bloatPRti{[t_i, t_{i+1}]} \; .
	\hspace{26pt}	
	\square
\end{align*}

The Minkowski sum in~\eqref{eq:HRPR_start} typically increases the representation size of the resulting set.
To counteract this growth, we enclose~$\PRset{}$ in~\eqref{eq:HRPR_start}
by a set specified by less parameters which over-approximates the original set
%We capture this over-approximation
by the error $\bloatSet{}$ accumulating
%which accumulates in the same way
as the error $\bloatPR{}$ of $\PRset{}$:
\begin{align}
	&\bloatSetti{[0,t_{i+1}]} = \bloatSetti{[0,t_i]} + \bloatSetti{[t_i,t_{i+1}]} \; ,
		\label{eq:bloatSetti_prop} \\
	&\text{with} \quad \bloatSetti{[0,0]} = 0 \; . \nonumber
\end{align}
Similarly to \propref{prop:bloatPRtiadm},
we define an admissible error $\bloatSettiadm{i}$: %for each step:
\begin{equation} \label{eq:bloatSettiadm}
	\bloatSetti{[t_i,t_{i+1}]} \leq \bloatSettiadm{i}
		:= \frac{\bloatSetmax{} - \bloatSetti{[0,t_i]}}{t_f - t_i} \, \Delta t_i \; .
\end{equation}
If we do not over-approximate the set representation
in the $i$-th time step, we have $\bloatSetti{[t_i,t_{i+1}]} = 0$.

% ------------------------------------------------------------------------------------
% ------------------------------------------------------------------------------------

\subsection{Adaptive Parameter Tuning}
\label{ssec:tuning}

We now present our generic framework for adaptive parameter tuning
in \algref{alg:full}.
The two subroutines, \algref{alg:postTSTT} and \algref{alg:propInhom},
will be presented thereafter.
First, we initialize $\Delta t$ and the scaling factor $\mu$ (\lineref{alg:full:inits}),
as well as the accumulating errors $\bloatPR{}$ and $\bloatSet{}$
(\lineref{alg:full:initbloat})
for the first iteration of the main loop
(\linerangeref{alg:full:mainloop_start}{alg:full:mainloop_end}).
In every iteration, we first obtain the parameters
$\Delta t_i$ and $\AppModPidx{i}$ (\lineref{alg:full:postTSTT})
from \algref{alg:postTSTT}.
Then, we calculate the reachable sets
$\HRti{[0,\Delta t_i]}$ and $\PRti{[0,\Delta t_i]}$ (\lineref{alg:full:HRPR0Deltati}),
which are subsequently used to obtain $\HRti{[t_i,t_{i+1}]}$
and $\PRti{[0,t_{i+1}]}$ (\lineref{alg:full:HRPR_prop}).
In \algref{alg:propInhom}, we obtain the parameters $\SetReprPidx{i}$,
which are used to recalculate $\PRti{[0,t_{i+1}]}$ (\lineref{alg:full:propInhom}).
At the end of each step, we add the homogeneous and inhomogeneous solution
to obtain the reachable set $\Rti{[t_i, t_{i+1}]}$ (\lineref{alg:full:Rti}).
Finally, after the calculation of the reachable set for each time interval,
the reachable set of the whole time horizon $\Rti{[0,t_f]}$
is obtained by unifying partial sets (\lineref{alg:full:R0T}).

\begin{theorem} \label{thm:alg:full}
\algref{alg:full} terminates and the overall error
%due to the sets $\HRplusset{}$, $\PRplusset{}$, and the over-approximation of the set representation
stays below a user-defined error bound $\bloatmax{} \in \R{+}$,
assuming the reduction of the set representation in each step is optional.
\end{theorem}

\FloatBarrier
%\vspace{0.3cm}

\begin{algorithm}[h!tb]
	\caption{Fully automated parameter tuning} \label{alg:full}
	\textbf{Input:} $\bloatHRmax{}, \bloatPRmax{}, \bloatSetmax{}, t_f$
	
	\textbf{Output:} $\Rti{[0,t_f]}$
	\begin{algorithmic}[1]
		\State $t \gets 0, i \gets 0$
%			\Comment{Initialize time, iteration counter}
		\State $\factor{} \gets 0.9, \Delta t_{-1} \gets t_f \, \factor{}$
%			\Comment{Initialize scaling factor, $\Delta t$}
			\label{alg:full:inits}
		\State $\bloatPRti{[0,0]} \gets 0$, $\bloatSetti{[0,0]} \gets 0$
%			\Comment{Initialize errors}
			\label{alg:full:initbloat}
%		\State $\Delta t_0, \AppModPidx{0}, \bloatPRti{[0,\Delta t_0]} \gets$ \algref{alg:postTSTT}
%			\label{alg:full:init}
%		\State calc. $\HRti{[0,\Delta t_0]}, \PRti{[0,\Delta t_0]}$ \label{alg:full:HRPRDeltat0}
%		\State $\Rti{[t_0, t_1]} = \HRti{[0,\Delta t_0]} \oplus \PRti{[0,\Delta t_0]}$
%			 \label{alg:full:Rt0t1}
%		\State $t \gets t + \Delta t_0$ \Comment{Update current time}
		\While{$t < T$} %\Comment{Loop over entire time horizon}
			\label{alg:full:mainloop_start}
%			\State $i \gets i + 1$ \Comment{Update iteration counter}
			\State $\Delta t_i, \AppModPidx{i}, \bloatPRti{[0,t_{i+1}]} \gets$ \algref{alg:postTSTT}
%				\Comment{Obtain $\Delta t$ and $\AppModP{}$ for current step}
				\label{alg:full:postTSTT}
			\State calc. $\HRti{[0,\Delta t_i]}, \PRti{[0,\Delta t_i]}$ using $\Delta t_i$, $\AppModPidx{i}$
%				\Comment{Compute auxiliary sets}
				\label{alg:full:HRPR0Deltati}
			\State calc. $\HRti{[t_i,t_{i+1}]}, \PRti{[0,t_{i+1}]}$ acc. to \eqref{eq:HRPR_start}
%				\Comment{Propagate auxiliary sets}
				\label{alg:full:HRPR_prop}
			\State $\PRti{[0,t_{i+1}]}, \bloatSetti{[0,t_{i+1}]} \gets$ \algref{alg:propInhom}
%				\Comment{$\SetReprPidx{i}$ adapted internally} 
				\label{alg:full:propInhom}
			\State $\Rti{[t_i, t_{i+1}]} = \HRti{[t_i,t_{i+1}]} \oplus \PRti{[0,t_{i+1}]}$
%				\Comment{Reachable set of current step}
				\label{alg:full:Rti}
			\State $t \gets t + \Delta t_i$, $i \gets i+1$
%				\Comment{Update time and iteration counter}
				\label{alg:full:increaset}
%			\State $i \gets i + 1$ \Comment{Update iteration counter}
		\EndWhile \label{alg:full:mainloop_end}
		\State $\Rti{[0,t_f]} \gets \bigcup_{j=0}^{i-1} \Rti{[t_j, t_{j+1}]}$
%			\Comment{Union of all reachable sets}
			\label{alg:full:R0T}
	\end{algorithmic}
\end{algorithm}
%
%\vspace{-0.8cm}
%\FloatBarrier
%
\begin{proof}
%\noindent \textit{Proof:}
We first show that the algorithm terminates.
The while-loop in \algref{alg:postTSTT} always terminates if
\begin{equation*}
	\forall i: \bloatPRti{[t_i, t_{i+1}]} \leq \bloatPRtiadm{i}
	\; , \; \bloatHRti{[t_i, t_{i+1}]} \leq \bloatHRmax{}	\; .
\end{equation*}
The first inequality can be satisfied as shown in the proof of \propref{prop:bloatPRtiadm}.
From \lineref{alg:postTSTT:decreaseDeltat} in \algref{alg:postTSTT}
and under \assref{ass:bloatHRPR}, we have
$\forall i: \Delta t_i \rightarrow 0 \Rightarrow \bloatHRti{[t_i,t_{i+1}]} \rightarrow 0$.
Thus, every bound $\bloatHRmax{}$ is satisfiable.
The while-loop in \algref{alg:propInhom} terminates as we will
either exceed the admissible error bound $\bloatSettiadm{i}$
during the continuous reduction or exit the loop
if the set parameters $\SetReprP{}$ have been reduced as much as possible.
Finally, the main loop (\linerangeref{alg:full:mainloop_start}{alg:full:mainloop_end})
terminates as the time~$t$ monotonically increases (\lineref{alg:full:increaset}),
eventually reaching~$t_f$.

To prove that the error stays below a user-defined bound,
we split the error $\bloatmax{}$
in $\bloatmax{} = \bloatHRmax{} + \bloatPRmax{} + \bloatSetmax{}$.
All individual bounds %$\bloatHRmax{}$, $\bloatPRmax{}$, and $\bloatSetmax{}$
are satisfiable:
We have already shown the satisfiability of any $\bloatHRmax{}$
in the beginning of this proof.
\propref{prop:bloatPRtiadm} shows that any limit $\bloatPRmax{}$ is satisfiable.
Lastly, any bound~$\bloatSetmax{}$ can be satisfied since %each step
we can choose to not over-approximate the set representation.
%$~$ \hfill $~$ $\square$
%\hfill $\square$
\end{proof}

\noindent We now present two algorithms performing the adaptive parameter tuning:
As an overview, \algref{alg:postTSTT} adapts
the parameters $\Delta t_i$ and $\AppModPidx{i}$,
while \algref{alg:propInhom} adapts the parameters $\SetReprPidx{i}$
concurrently to the propagation of $\PRset{}$.

% postTSTT
We first present \algref{alg:postTSTT}:
As the main idea, every $\AppModP{}$ is checked for an iteratively decreasing $\Delta t_i$
until the error bounds $\bloatHRmax{}$ and $\bloatPRtiadm{i}$ are satisfied. 
Initially, we increase the time step size by division with the factor $\factor{}$ %$\in~(0,1)$
(\lineref{alg:postTSTT:startDeltat}).
Furthermore, we reset~$\AppModPidx{i}$ to its coarsest setting (\lineref{alg:postTSTT:startDeltat})
and calculate the admissible error $\bloatPRtiadm{i}$ (\lineref{alg:postTSTT:bloatPRtiadm}).
In the loop (\linerangeref{alg:postTSTT:loop_start}{alg:postTSTT:loop_end}),
we increment $\AppModPidx{i}$ (\lineref{alg:postTSTT:incrAppModP})
and calculate the errors $\bloatHRti{[t_i,t_{i+1}]}$ and $\bloatPRti{[t_i,t_{i+1}]}$
based on the error sets $\HRplusti{[t_i,t_{i+1}]}$ and $\PRplusti{[t_i,t_{i+1}]}$
(\linerangeref{alg:postTSTT:errors_start}{alg:postTSTT:errors_end}).
These errors are then compared to their respective error bounds
%(\linerangeref{alg:postTSTT:condition_start}{alg:postTSTT:loop_end}).
(\lineref{alg:postTSTT:loop_end}).
If $\AppModPidx{i}$ reaches $\AppModPmax{}$ (\lineref{alg:postTSTT:AppModPmax}),
we decrease $\Delta t_i$, reset $\AppModPidx{i}$ (\lineref{alg:postTSTT:decreaseDeltat}),
and recalculate the admissible bound $\bloatPRtiadm{i}$
(\lineref{alg:postTSTT:updatebloatPRtiadm})
before restarting the loop (\lineref{alg:postTSTT:continue}).
After the loop is finished, the accumulated error $\bloatPRti{[0,t_{i+1}]}$
is computed (\lineref{alg:postTSTT:bloatPRti_acc}).

\FloatBarrier
%\vspace{-0.3cm}

\begin{algorithm}[h!tb]
	\caption{Adapt values for $\Delta t, \AppModP{}$} \label{alg:postTSTT}
	\textbf{Input:} $\bloatHRmax{}, \bloatPRmax{}, \bloatPRti{[0,t_i]}, \Delta t_{i-1}, \factor{}, t_f$
	
	\textbf{Output:} $\Delta t_i, \AppModPidx{i}, \bloatPRti{[0,t_{i+1}]}$
	\begin{algorithmic}[1]
		\State $\Delta t_i \gets \frac{\Delta t_{i-1}}{\factor{}}$, $\resetOp{\AppModPidx{i}}$
%			\Comment{Increase $\Delta t$, reset $\AppModPidx{i}$}
			\label{alg:postTSTT:startDeltat}
%		\State $\resetOp{\AppModPidx{i}}$ 
%			\Comment{Reset $\AppModPidx{i}$} \label{alg:postTSTT:startAppModP}
		\State calc. $\bloatPRtiadm{i}$ acc. to \eqref{eq:bloatPRtiadm}
%			\Comment{Admissible error for current step}
			\label{alg:postTSTT:bloatPRtiadm}
		\State \textbf{do}
%			\Comment{Loop until error bounds satisfied}
			\label{alg:postTSTT:loop_start}
			\State $\quad$ $\incrOp{\AppModPidx{i}}$
%				\Comment{Increment $\AppModPidx{i}$}
				\label{alg:postTSTT:incrAppModP}
			\State $\quad$ \textbf{if} $\AppModP{} \geq \AppModPmax{}$
%				\Comment{Maximum of $\AppModP{}$ reached}
				\label{alg:postTSTT:AppModPmax}
				\State $\qquad$ $\Delta t_i \gets \Delta t_i \factor{}$, $\resetOp{\AppModPidx{i}}$
%					\Comment{Decrease $\Delta t$, reset $\AppModPidx{i}$}
					\label{alg:postTSTT:decreaseDeltat}
%				\State $\qquad$ $\resetOp{\AppModPidx{i}}$ \Comment{Reset $\AppModPidx{i}$}
%					\label{alg:postTSTT:resetAppModP}
				\State $\qquad$ calc. $\bloatPRtiadm{i}$ acc. to \eqref{eq:bloatPRtiadm}
%					\Comment{Adapt admissible error} 					
					\label{alg:postTSTT:updatebloatPRtiadm}
				\State $\qquad$ \textbf{continue} %\Comment{Restart loop}
					\label{alg:postTSTT:continue}
			\State $\quad$ \textbf{end if}
%			\If{$\AppModP{} \geq \AppModPmax{}$} \label{alg:postTSTT:AppModPmax}
%				\State $\Delta t_i \gets \Delta t_i \factor{}$
%					\Comment{Decrease $\Delta t$} \label{alg:postTSTT:decreaseDeltat}
%				\State $\resetOp{\AppModPidx{i}}$ \Comment{Reset $\AppModPidx{i}$}
%					\label{alg:postTSTT:resetAppModP}
%				\State calc. $\bloatPRtiadm{i}$ acc. to \eqref{eq:bloatPRtiadm}
%					\label{alg:postTSTT:updatebloatPRtiadm}
%				\State \textbf{continue} \Comment{Restart loop} \label{alg:postTSTT:continue}
%			\EndIf
			\State $\quad$ calc. $\HRplusti{[t_i,t_{i+1}]}, \bloatHRti{[t_i, t_{i+1}]}$
%				\Comment{Homogeneous error of current step}
				\label{alg:postTSTT:errors_start}
			\State $\quad$ calc. $\PRplusti{[t_i,t_{i+1}]}, \bloatPRti{[t_i, t_{i+1}]}$
%				\Comment{Inhomogeneous error of current step}
				\label{alg:postTSTT:errors_end}
%		\State \textbf{while} $\bloatHRti{[t_i, t_{i+1}]} > \bloatHRmax{}$ ...
%			\label{alg:postTSTT:condition_start} \\
%			$~~$\textbf{and} $\bloatPRti{[t_i, t_{i+1}]} > \bloatPRtiadm{i}$
%				\label{alg:postTSTT:loop_end}
		\State \textbf{while} $\bloatHRti{[t_i, t_{i+1}]} > \bloatHRmax{}
			\wedge \bloatPRti{[t_i, t_{i+1}]} > \bloatPRtiadm{i}$
				\label{alg:postTSTT:loop_end}
		\State calc. $\bloatPRti{[0,t_{i+1}]}$ acc. to \eqref{eq:bloatPRti_acc}
%			\Comment{Accumulated error of the inhom. solution}			
			\label{alg:postTSTT:bloatPRti_acc}
%		\State \textbf{end while} \label{alg:postTSTT:loop_end}
%		\While{$\bloatHRti{[t_i, t_{i+1}]} \geq \bloatHRmax{}$ ... \label{alg:postTSTT:loop_start} \\
%			$~~$\textbf{and} $\bloatPRti{[t_i, t_{i+1}]} \geq \bloatPRtiadm{i}$}
%				\label{alg:postTSTT:condition_end}
%		\EndWhile \label{alg:postTSTT:loop_end}
	\end{algorithmic}
\end{algorithm}

\FloatBarrier
%\vspace{-0.3cm}

The adaptation of~$\SetReprP{}$ is shown in \algref{alg:propInhom}:
In order to decrease the computation time, the number of parameters
describing $\PRset{}$ is iteratively reduced from a high-precision representation
towards lower precision.
%until either the admissible error is reached
%or the parameters $\SetReprP{}$ are reduced to their minimum.
First, we calculate the admissible error~$\bloatSettiadm{i}$
(\lineref{alg:propInhom:bloatSettiadm}) for the given~$\Delta t_i$.
The operator $\paramOp{\SetReprP{}}$ extracts
the number of stored parameters~$\SetReprPidx{i}$
from $\PRti{[0,t_{i+1}]}$ (\lineref{alg:propInhom:readSetReprP}).
Inside the loop (\linerangeref{alg:propInhom:loop_start}{alg:propInhom:loop_end}),
we iteratively decrease $\SetReprPidx{i}$ (\lineref{alg:propInhom:decrSetReprP})
and over-approximate $\PRti{[0,t_{i+1}]}$
by reducing the number of stored parameters to $\SetReprPidx{i}$
(\lineref{alg:propInhom:reduce}).
From the difference between the original set and the set over-approximated by reduction,
we calculate the induced error~$\bloatSetti{[t_i,t_{i+1}]}$ (\lineref{alg:propInhom:bloatSetti})
and compare it to the admissible error (\lineref{alg:propInhom:loop_end}).
After the loop,
the accumulated error for the set representation $\bloatSetti{[0,t_{i+1}]}$
is computed (\lineref{alg:propInhom:bloatSetti_prop}).

\FloatBarrier
%\vspace{-0.3cm}

\begin{algorithm}[h!tb]
	\caption{Adapt values for $\SetReprP{}$} \label{alg:propInhom}
	\textbf{Input:} $\bloatSetmax{}, \bloatSetti{[0,t_i]}, \PRti{[0,t_{i+1}]}, \Delta t_i, t_f$
	
	\textbf{Output:} $\PRti{[0,t_{i+1}]}, \bloatSetti{[0,t_{i+1}]}$
	\begin{algorithmic}[1]
		\State calc. $\bloatSettiadm{i}$ acc. to \eqref{eq:bloatSettiadm}
%			\Comment{Admissible bound for current step}
			\label{alg:propInhom:bloatSettiadm}
		%\State $\PRti{[0,t_{i+1}]} \gets \PRti{[0,t_i]} + \PRti{[t_i,t_{i+1}]}$
		\State $\SetReprPidx{i} \gets \paramOp{\PRti{[0,t_{i+1}]}}$
%			\Comment{Extract number of stored parameters}
			\label{alg:propInhom:readSetReprP}
		\State \textbf{do}
%			\Comment{Loop until admissible error reached}			
			\label{alg:propInhom:loop_start}
			\State $\quad$ $\decrOp{\SetReprPidx{i}}$
%				\Comment{Decrease number of parameters}
				\label{alg:propInhom:decrSetReprP}
			\State $\quad$ $\redOp{\PRti{[0,t_{i+1}]}}{\SetReprPidx{i}}$
%				\Comment{Reduce $\PRti{[0,t_{i+1}]}$}
				\label{alg:propInhom:reduce}
			\State $\quad$ calc. $\bloatSetti{[t_i,t_{i+1}]}$
%				\Comment{Error due to order reduction}
				\label{alg:propInhom:bloatSetti}
		\State \textbf{while} $\bloatSetti{[t_i,t_{i+1}]} < \bloatSettiadm{i}
			\vee \decrOp{\SetReprP{}} = \SetReprP{}$
			\label{alg:propInhom:loop_end}
		\State calc. $\bloatSetti{[0,t_{i+1}]}$ acc. to \eqref{eq:bloatSetti_prop}
%			\Comment{Accumulated error due to order reduction}
			\label{alg:propInhom:bloatSetti_prop}
%		\State \textbf{end while}
%		\While{$\bloatSetti{[t_i,t_{i+1}]} > \bloatSettiadm{i}$}
%			\label{alg:propInhom:loop_start}
%		\EndWhile \label{alg:propInhom:loop_end}
	\end{algorithmic}
\end{algorithm}

\FloatBarrier
%\vspace{-0.3cm}

The presented framework for adaptive parameter tuning
can be applied to any reachability algorithm.
In the next section, we will provide an example implementation
and prove the assumptions underlying \thmref{thm:alg:full}.

% -------------------------------------------------------
% IMPLEMENTATION

\section{Implementation}
\label{sec:implementation}

In this section, we present an implementation of our framework
and validate our assumptions using a concrete reachability algorithm and a concrete set representation.

\subsection{Computation of Reachable Sets}
\label{ssec:setprop}

We over-approximate the exponential matrix~$e^{At}$
by a finite number of Taylor terms~$\eta \in \naturalsof{+}$
and an interval matrix~$\E{}$ enclosing the remainder \cite[Prop. 2]{Althoff2011a}.
\begin{align}
	e^{A \Delta t} &\in \sum_{k=0}^{\eta} \frac{(A \Delta t)^k}{k!} \oplus \Eof{\Delta t, \eta} \; ,
		\label{eq:eAt} \\
	\Eof{\Delta t, \eta} &= [-\Eabsof{\Delta t, \eta}, \Eabsof{\Delta t, \eta}]\label{eq:E}  \\
%	\begin{split} \label{eq:Eabs}
%		\text{with} \quad \Eabsof{\Delta t, \eta}
%		&= \Big|e^{|A|\Delta t} - \sum_{k=0}^\eta \frac{1}{k!} \big(|A| \Delta t\big)^k \Big| \\
%		&= \Big| \sum_{k=\eta + 1}^\infty \frac{1}{k!} \big(|A| \Delta t\big)^k \Big| \; .
%	\end{split}
%	\text{with} \quad \Eabsof{\Delta t, \eta}
%		&= \Big|e^{|A|\Delta t} - \sum_{k=0}^\eta \frac{1}{k!} \big(|A| \Delta t\big)^k \Big|
%		= \Big| \sum_{k=\eta + 1}^\infty \frac{1}{k!} \big(|A| \Delta t\big)^k \Big| \; . \label{eq:Eabs}
	\text{with} \quad \Eabsof{\Delta t, \eta}
		&= \Big| \sum_{k=\eta + 1}^\infty \frac{1}{k!} \big(|A| \Delta t\big)^k \Big| \; . \label{eq:Eabs}
\end{align}
To cover all trajectories in a time interval spanned by~$\Delta t$,
we introduce the terms $\F{}$ \cite[Sec. 4]{Althoff2011a}
and $\inputCorr{}$ \cite[Sec. 3.2.2]{Althoff2010a}:
%\cite[Thm. 3]{Althoff2007c}
%which accounts for the trajectory curvature:
%To account for the time-varying input~$u(t)$ over a given time interval~$\Delta t$,
%we introduce the term~$\inputCorr{}$ \cite[Sec. 3.2.2]{Althoff2010a}:
%
\begin{align}
	\Fof{\Delta t, \eta} &= \bigoplus_{k=2}^\eta [(k^{\frac{-k}{k-1}}
		- k^{\frac{-1}{k-1}}) \Delta t^k,0] \frac{A^k}{k!}
		\oplus \Eof{\Delta t, \eta} \; , \label{eq:F} \\
	\begin{split} \label{eq:inputCorr}
		\inputCorrof{\Delta t, \eta} &= \bigoplus_{k=2}^{\eta+1}
			[(k^{\frac{-k}{k-1}} - k^{\frac{-1}{k-1}}) \Delta t^k,0] \frac{A^{k-1}}{k!} \\
			&\qquad \oplus \Eof{\Delta t, \eta} \, \Delta t \; .
	\end{split}
\end{align}
%\begin{equation} \label{eq:inputCorr}
%	\inputCorrof{\Delta t, \eta} = \bigoplus_{k=2}^{\eta+1}
%		[(k^{\frac{-k}{k-1}} - k^{\frac{-1}{k-1}}) \Delta t^k,0] \frac{A^{k-1}}{k!}
%		\oplus \Eof{\Delta t, \eta} \, \Delta t \; .
%\end{equation}

The reachable sets for the homogeneous and inhomogeneous solution,
generally defined in~\eqref{eq:HRtieqplus} and~\eqref{eq:PRtieqplus},
are computed by~\cite[Sect. 3.2.1-3.2.2]{Althoff2010a}
\begin{align}
	\HRti{[0, \Delta t]} &=
		\underbrace{\CH{\initset{}, e^{A \Delta t} \initset{}}}_{= \, \HReqtiunderbrace{[0,\Delta t]}} \nonumber \\
		&\qquad \oplus \underbrace{\Fof{\Delta t, \eta} \, \initset{}
		\oplus \inputCorrof{\Delta t, \eta} \, \uTrans{}}%ctd underbrace
			_{= \, \HRplustiunderbrace{[0,\Delta t]}} \;, \label{eq:HRti_cora} \\
%	\HRti{[0, \Delta t]} &=
%		\underbrace{\CH{\initset{}, e^{A \Delta t} \initset{}}}_{= \, \HReqtiunderbrace{[0,\Delta t]}}
%		\oplus \underbrace{\Fof{\Delta t, \eta} \, \initset{}
%		\oplus \inputCorrof{\Delta t, \eta} \, \uTrans{}}%ctd underbrace
%			_{= \, \HRplustiunderbrace{[0,\Delta t]}} , \label{eq:HRti_cora} \\
	\PRti{[0, \Delta t]} &= \underbrace{\sum_{k=0}^{\eta}
		\Big( \frac{A^k \Delta t^{k+1}}{(k+1)!} \Big) \, \inputset{}}_{= \, \PReqtiunderbrace{[0,\Delta t]}}
		\oplus \underbrace{\Eof{\Delta t, \eta} \, \Delta t \, \inputset{}}_{= \, \PRplustiunderbrace{[0,\Delta t]}} ,
			\label{eq:PRti_cora}
\end{align}
with~$\uTrans{}$ being the center of the input set~$\inputset{}$.
We include the term~$\inputCorrof{\Delta t, \eta} \, \uTrans{}$ in the homogeneous solution
as it only covers the current time interval and therefore does not accumulate over time.
In the next section, we will verify the applicability of
\eqref{eq:HRti_cora} and \eqref{eq:PRti_cora} for our adaptive framework.

\subsection{Verification of Assumptions}
\label{ssec:verass}

We now want to verify \assref{ass:bloatHRPR}
and \propref{prop:bloatPRtiadm}.
%for the reachable sets introduced in the previous subsection.
%
To obtain the error $\bloatHRti{[t_i, t_{i+1}]}$
of the homogeneous solution $\HRti{[t_i, t_{i+1}]}$,
we multiply \eqref{eq:HRti_cora} by $e^{A t_i}$ as required in \eqref{eq:HRPR_start}
and insert the result in \eqref{eq:bloatHRti}, which yields
%\begin{equation} \label{eq:bloatHRti_cora}
%	\bloatHRti{[t_i, t_{i+1}]} =
%		\errOp{ e^{A t_i} \, \Fof{\Delta t_i, \eta_i} \, \initset{}
%			\oplus \, e^{A t_i} \, \inputCorrof{\Delta t_i, \eta_i} \, \uTrans{} } \; .
%\end{equation}
\begin{align}
	\begin{split} \label{eq:bloatHRti_cora}
	\bloatHRti{[t_i, t_{i+1}]} &=
		\errOp{ e^{A t_i} \, \Fof{\Delta t_i, \eta_i} \, \initset{} \\
			&\qquad \qquad \oplus \, e^{A t_i} \, \inputCorrof{\Delta t_i, \eta_i} \, \uTrans{} } \; .
	\end{split}
\end{align}
\begin{proposition} \label{prop:bloatHRti_cora}
The error $\bloatHRti{[t_i,t_{i+1}]}$ in \eqref{eq:bloatHRti_cora}
converges to 0 for $\Delta t_i \rightarrow 0$
and therefore satisfies \assref{ass:bloatHRPR}.
\end{proposition}
\begin{proof}
%\textit{Proof:}
Since the operation $\errOp{\Sarb{}}$ returns the enclosing radius
of the interval over-approximation of a set $\Sarb{}$,
it suffices to show that the volume of $\Sarb{}$
converges to 0 for $\Delta t_i \rightarrow 0$.
This is the case if the interval matrices $\F{}$ and $\inputCorr{}$
converge to $[\zeroMat{n}, \zeroMat{n}]$:
Using \eqref{eq:inputCorr}, we yield
$\lim_{\Delta t \rightarrow 0} \inputCorrof{\Delta t, \eta} = [\zeroMat{n}, \zeroMat{n}]$.
%\begin{align*}
%	\lim_{\Delta t \rightarrow 0}
%	\inputCorrof{\Delta t, \eta} &\overset{\eqref{eq:inputCorr}}{=}
%		\bigoplus_{k=2}^{\eta+1} [(k^{\frac{-k}{k-1}} - k^{\frac{-1}{k-1}}) \, 0^k,0] \frac{A^{k-1}}{k!}
%		\oplus \Eof{\Delta t, \eta} \, 0 \\
%	&= [\zeroMat{n}, \zeroMat{n}] \; .
%\end{align*}
%\begin{equation*}
%	\lim_{\Delta t \rightarrow 0}
%	\inputCorrof{\Delta t, \eta} \overset{\eqref{eq:inputCorr}}{=}
%		\bigoplus_{k=2}^{\eta+1} [(k^{\frac{-k}{k-1}} - k^{\frac{-1}{k-1}}) \, 0^k,0] \frac{A^{k-1}}{k!}
%		\oplus \Eof{\Delta t, \eta} \, 0 
%	= [\zeroMat{n}, \zeroMat{n}] \; .
%\end{equation*}
%
For $\F{}$, we have that
$\lim_{\Delta t \rightarrow 0} \Eof{\Delta t, \eta} \overset{\eqref{eq:E}}{=} [\zeroMat{n},\zeroMat{n}]$ since
$\lim_{\Delta t \rightarrow 0} \Eabsof{0, \eta} \overset{\eqref{eq:Eabs}}{=} \zeroMat{n}$.
%\begin{align}
%	&\lim_{\Delta t \rightarrow 0} \Eof{\Delta t, \eta} \overset{\eqref{eq:E}}{=}
%		[-\Eabsof{0, \eta}, \Eabsof{0, \eta}] = [\zeroMat{n},\zeroMat{n}] \nonumber \\
%	&\, \text{since} \quad \Eabsof{0, \eta} \overset{\eqref{eq:Eabs}}{=}
%		\Big|\sum_{k=\eta + 1}^{\infty} \frac{1}{k!} \big(|A| \, 0\big)^k \Big| = \zeroMat{n} \; . \label{eq:Eabs0}
%\end{align}
By plugging this in \eqref{eq:F}, we immediately see that
$\Fof{\Delta t, \eta} = [\zeroMat{n}, \zeroMat{n}]$ for $\Delta t \rightarrow 0$.
%
%\begin{align*}
%%	\hspace{17pt}
%	\lim_{\Delta t \rightarrow 0}
%		\Fof{\Delta t, \eta}
%	&\overset{\eqref{eq:F}}{=} \bigoplus_{k=2}^\eta
%		[(k^{\frac{-k}{k-1}} - k^{\frac{-1}{k-1}}) \, 0^k,0] \frac{A^k}{k!} \oplus \Eof{0, \eta} \\
%	&= [\zeroMat{n}, \zeroMat{n}] \; .
%	\hspace{120pt} \blacksquare
%\end{align*}
%\begin{equation*}
%	\hspace{22pt} \lim_{\Delta t \rightarrow 0}
%		\Fof{\Delta t, \eta}
%	\overset{\eqref{eq:F}}{=} \bigoplus_{k=2}^\eta
%		[(k^{\frac{-k}{k-1}} - k^{\frac{-1}{k-1}}) \, 0^k,0] \frac{A^k}{k!} \oplus \Eof{0, \eta}
%	= [\zeroMat{n}, \zeroMat{n}] \; . \hspace{22pt} \square
%\end{equation*}
%
%Hence, we can conclude that \eqref{eq:bloatHRti_cora} satisfies \assref{ass:bloatHR}.
\end{proof}

To obtain the error $\bloatPRti{[t_i, t_{i+1}]}$
of the inhomogeneous solution $\PRti{[t_i, t_{i+1}]}$,
we multiply \eqref{eq:PRti_cora} by $e^{A t_i}$ as required in \eqref{eq:HRPR_start}
and insert the result in \eqref{eq:bloatPRti}, which gives us
\begin{equation} \label{eq:bloatPRti_cora}
	\bloatPRti{[t_i, t_{i+1}]} =
		\errOp{ e^{A t_i} \, \Eof{\Delta t_i, \eta_i} \, \Delta t_i \, \inputset{} } \; .
\end{equation}
\begin{proposition} \label{prop:bloatPRti_cora}
The error $\bloatPRti{[t_i,t_{i+1}]}$ in \eqref{eq:bloatPRti_cora}
converges to 0 for $\Delta t_i \rightarrow 0$
and therefore satisfies \assref{ass:bloatHRPR}.
\end{proposition}
\begin{proof}
Equivalently to the proof of \propref{prop:bloatHRti_cora},
we show that the volume of the resulting set converges to 0 for ${\Delta t_i \rightarrow 0}$.
Since $\Delta t_i$ appears as a multiplicative factor, it holds that
$\lim_{\Delta t_i \rightarrow 0} \, e^{A t_i} \, \Eof{\Delta t_i, \eta_i} \, \Delta t_i \, \inputset{}
= [\zeroMat{n}, \zeroMat{n}]$.
$~$ \hfill $~$ %$\square$
\end{proof}
We introduce the following lemma for the subsequent derivations:
\begin{lemma} \label{lmm:E}
The size of $\Eof{\Delta t, \eta}$ decreases superlinearly with respect to $\Delta t$:
\begin{equation*}
	\forall \varphi \in (0,1): \Eof{\varphi \Delta t, \eta} \subseteq \varphi \, \Eof{\Delta t, \eta} \; .
\end{equation*}
\end{lemma}

%\begin{proof}
\noindent \textit{Proof.}
Using $\Eof{\Delta t, \eta} \overset{\eqref{eq:E}}{=} [-\Eabsof{\Delta t, \eta},\Eabsof{\Delta t, \eta}]$,
it is sufficient to show that each entry of $\Eabsof{\Delta t, \eta}$ decreases superlinearly
with respect to $\Delta t$:
\begin{align*}
	\hspace{20pt}
	&\forall \varphi \in (0,1): \Eabsof{\varphi \Delta t, \eta} \overset{\eqref{eq:Eabs}}{=}
	\Big| \sum_{k = \eta + 1}^{\infty} \frac{1}{k!} \big(|A| \varphi \Delta t\big)^k \Big| \\
	&\qquad \leq \varphi \, \Big| \sum_{k = \eta + 1}^{\infty} \frac{1}{k!} \big(|A| \Delta t\big)^k \Big|
	= \varphi \, \Eabsof{\Delta t, \eta} \; .
	\hspace{18pt} \square
\end{align*}
%\end{proof}

\begin{theorem}
The error $\bloatPRti{[t_i,t_{i+1}]}$ in \eqref{eq:bloatPRti_cora}
decreases superlinearly according to \defref{def:superlinear}:
\begin{equation*}
	\forall \varphi \in (0,1): \bloatPRti{[t_i, t_i + \varphi \Delta t_i]}
		\leq \varphi \, \bloatPRti{[t_i, t_i + \Delta t_i]} \; .
\end{equation*}
\end{theorem}

%\begin{proof}
\textit{Proof:}
Since $\errOp{\cdot}$ is defined by the enclosing radius~$r$
according to \propref{prop:bloat},
it suffices to show that the enclosing radius decreases
superlinearly with respect to $\Delta t_i$:
\begin{equation*} % \label{eq:PRplus_decrease}
	\hspace{-0pt}
	\forall \varphi \in (0,1)\!\!: r ( \PRplusti{[t_i, t_i + \varphi \Delta t_i]} )
	\leq \varphi \, r ( \PRplusti{[t_i, t_i + \Delta t_i]} ) .
%	\forall \varphi \in (0,1): r \big( \PRplusti{[t_i, t_i + \varphi \Delta t_i]} \big)
%	\leq \varphi \, r \big( \PRplusti{[t_i, t_i + \Delta t_i]} \big) .
\end{equation*}
%
%The condition \eqref{eq:PRplus_decrease} is satisfied if
This condition is satisfied if
\begin{equation*}
	\forall \varphi \in (0,1)\!: \PRplusti{[t_i, t_i + \varphi \Delta t_i]}
	\subseteq \varphi \, \PRplusti{[t_i, t_i + \Delta t_i]} \; .
\end{equation*}
Using the definition of $\PRplusti{[0, \Delta t]}$ in \eqref{eq:PRti_cora}
and \lmmref{lmm:E} it holds that
\begin{align*}
%	\hspace{10pt}
	&\PRplusti{[t_i, t_i + \varphi \Delta t_i]}
		\overset{\eqref{eq:PRti_cora}}{=} ( \varphi \, \Delta t_i ) \, e^{A t_i} \,
		\underbrace{\Eof{\varphi \Delta t_i, \eta_i}}_%
			{\overset{\text{\lmmref{lmm:E}}}{\subseteq} \,
			\varphi \, \Eof{\Delta t_i, \eta_i}} \, \inputset{} \subseteq \\
		&\varphi \, \big( \underbrace{\varphi \Delta t_i \, e^{A t_i} \, \Eof{\Delta t_i, \eta_i} \, \inputset{}}_%
			{\overset{\eqref{eq:PRti_cora}}{=} \,
			\varphi \, \PRplusti{[t_i, t_i + \Delta t_i]}} \big)
		\overset{\varphi \in (0,1)}{\subseteq} \varphi \, \PRplusti{[t_i, t_i + \Delta t_i]}. \;
%	\hspace{20pt}
	\square
\end{align*}
%\end{proof}

%Next, we consider \assref{ass:dominance} which states
%that the choice of $\Delta t$ dominates $\AppModP{}$.
%For the implementation in \sectref{ssec:setprop}, we have $\AppModP{} = \{\eta\}$
%which is the number of Taylor terms in~\eqref{eq:eAt}.
%We bound $\eta$ by a value $\etamax \in \naturals{}^+$.
%Furthermore, we have the operators
%$\resetOp{\AppModP{}}: \eta \gets 0$ and
%$\incrOp{\AppModP{}}: \eta \gets \eta + 1$.
%All of the terms used in \eqref{eq:HRti_cora} and \eqref{eq:PRti_cora}
%contain $\Delta t$ as a multiplicative factor,
%greatly influencing the size of $\E{}$, $\F{}$, and $\inputCorr{}$
%which contribute to the errors \eqref{eq:bloatHRti_cora} and \eqref{eq:bloatPRti_cora}.
%On the other hand, $\eta$ fixes the number of added terms
%in the sums for computing $\E{}$, $\F{}$, and $\inputCorr{}$,
%but does not directly influence their size.
The cut-off value $\etamax{}$ can be automatically obtained according to \cite{Bickart1968}
to truncate the power series in \eqref{eq:eAt}.
If $\etamax{}$ does not satisfy the current error bounds,
we proceed to smaller values for $\Delta t$
which are guaranteed to eventually satisfy the error by \thmref{thm:alg:full}.
%
%justified by the exponentially decreasing size
%of the added terms, for which we fix a conservative bound ($\etamax{} = 10$)
%above which we do not expect the precision to noticeably increase
%as the $\etamax{}$-th power shrinks the resulting term towards 0. \cite{Bickart1968}
%
%Finally, \assref{ass:noninterference} does not impose any restrictions since
%the computation of the reachable sets using $\Delta t$ and $\AppModP{}$
%precedes any subsequent reduction in $\SetReprP{}$.
%Consequently, any change in $\SetReprP{}$ can only affect the tightness
%of the reachable sets independently of $\Delta t$ and~$\AppModP{}$.

Following the above derivations, the presented implementation satisfies \thmref{thm:alg:full}.
Hence, it can be used to perform reachability analysis for LTI systems
while remaining below a user-defined error bound.

\subsection{Set Representation}
\label{ssec:setrepr}

It remains to choose a set representation.
The work in~\cite{Althoff2016c} shows that zonotopes are optimal
and that one should add support functions if the initial set is not a zonotope.
Since we only use zonotopes as initial sets, we only use them:
%The reachability algorithm presented in \eqref{eq:HRti_cora} and \eqref{eq:PRti_cora}
%applies three different set operations: Linear map, Minkowski sum, and convex hull.
%%\tabref{tab:setsoverview} shows the properties of commonly-used set representations,
%%which we discuss subsequently:
%We quickly discuss the properties of commonly-used set representations:
%Ellipsoids are not closed under Minkowski sum.
%The computational complexity of the Minkowski sum
%for polytopes in halfspace representation \cite[Def. 2.1]{Althoff2010a}
%is exponential in the dimension and therefore unusable in the general case.
%Polytopes in vertex representation \cite[Def. 2.2]{Althoff2010a}
%require strictly more operations than zonotopes
%as a zonotope has always more vertices than generators.
%Support functions only offer an implicit description of a set.
%Therefore, the best choice is to use zonotopes
%when $\initset{}$ and $\inputset{}$ are zonotopes \cite{Althoff2016c}.

\begin{definition} \label{def:zonotopes} (Zonotopes)
Given a center ${c \in \R{n}}$
and an arbitrary number ${\Gsize{} \in \naturals{}}$
of generator vectors ${g^{(1)}, ..., g^{(\Gsize{}{})} \in \R{n}}$,
a zonotope is defined as \cite[Def. 1]{Girard2005}
\begin{equation*} %\label{eq:zon}
%	\hspace{18pt}
	\mathcal{Z} = \Big\{ x \in \R{n} \, \Big| \, x
		= c + \sum_{i=1}^{\Gsize{}} \beta_i \cdot g^{(i)}, \, -1 \leq \beta_i \leq 1 \Big\} \; .
%	\hspace{12pt} \square
\end{equation*}
The order of a zonotope is $\zonOrder{} = \frac{\Gsize{}}{n}$. \hfill $\square$
\end{definition}

Following \defref{def:zonotopes}, we have that $\SetReprP{} = \zonOrder{}$.
%As discussed in~\sectref{ssec:errors},
%we require to reduce the order of the zonotope describing $\PRti{[0,t_{i+1}]}$
%using order reduction techniques \cite{Kopetzki2017}.
%Let us denote the number of generators of the unreduced set by $\Gsizeorig{}$.
%There are different ways of reducing the order of a zonotope \cite{Kopetzki2017}.
%We focus on methods which over-approximate a number of generators $\Gsizebox{} \leq \Gsize{}$
%by a tight multi-dimensional interval~$\Gbox{}$ as introduced in \cite[Sect. 3.4]{Girard2005}.
We iteratively decrease the zonotope order by decrementing
the total number of generators describing $\PRti{[0,t_{i+1}]}$
as shown in~\cite{Kopetzki2017},
which also provides us with $\bloatSetti{[t_i,t_{i+1}]}$.
Therefore, we define the operator
$\decrOp{\SetReprP{}}: n\Gsize{} \gets n(\Gsize{} - 1)$.
%This procedure is continued as long as
%the admissible error~$\bloatSettiadm{i}$ is satisfied.
%As for the errors $\bloatHR{}$ and $\bloatPR{}$,
%we have to define an error $\bloatSet{}$ for the set representation.
%The error due to the conversion is over-approximated by
%regarding the generators $\Gbox{}$ selected for reduction as an added term:
%%
%\begin{equation} \label{eq:errZon}
%	\bloatSetti{[t_i,t_{i+1}]} = \errOp{ \Gbox{} } \; .
%\end{equation}
%
%Naturally, we prioritize the selection of short generators
%in the reduction step as they induce the smallest error.
In the next section, we apply the presented implementation.% on numerical examples.

% -------------------------------------------------------
% NUMERICAL EXAMPLES

\section{Numerical Examples}
\label{sec:numex}

We have implemented our approach in \textit{MATLAB}.
To extensively test our approach, we perform several investigations:
First, we measure the computational overhead caused by our parameter tuning.
Next, \algref{alg:full} is evaluated on benchmark systems and compared to manually-tuned \params{}.
Finally, we compare \algref{alg:full} with a genetic algorithm searching for \params{}.
Since the genetic algorithm requires a lot of memory,
we used an Intel Xeon Gold 6136 3.00GHz processor and 768GB of DDR4 2666/3273MHz memory,
all other computations have been performed using an Intel i3 processor with 8GB memory.
For the genetic algorithm and the overhead measurement,
we used the subsequently introduced randomly-generated systems
because the results might vary depending on the investigated system.

\paragraph{Random Generation of Systems}
\label{par:randsys}

We randomly picked complex conjugate pairs of eigenvalues
from a uniform distribution over the range of $[-1,1]$ for the real part
and $[-\mathrm{i},\mathrm{i}]$ for the imaginary part
without loss of generality, since the characteristics of the solution
only depends on the ratio of real and imaginary values.
The state-space form for these eigenvalues is computed
and subsequently rotated by a random orthogonal matrix
of increasing sparsity for higher dimensions.
Furthermore, the initial set $\initset{}$ and the input set $\inputset{}$
are given by hypercubes centered at $(10,...,10)^T$ with edge length 0.5,
and $(1,...,1)^T$ with edge length 0.1, respectively.
Lastly, we set $\bloatmax{} = 0.05$ and $t_f=3$s.

%\FloatBarrier
%\vspace{-0.4cm}

\begin{figure}[h]
\begin{center}
	\setlength{\belowcaptionskip}{-15pt}
	\includegraphics[width = 0.49 \textwidth]{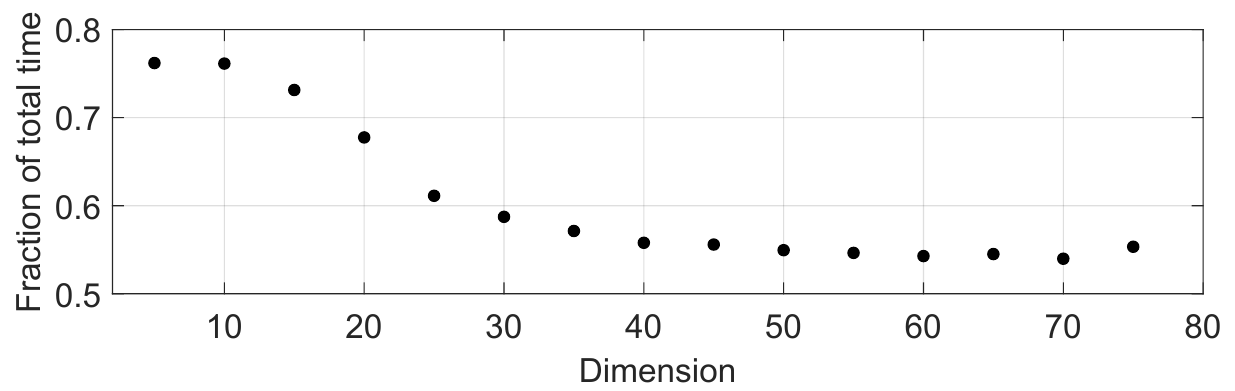}
	\caption{Time consumption by parameter tuning
		in relation to total execution time using
		50 randomly-generated systems per dimension.}
	\label{fig:overhead}
\end{center}
\end{figure}
% smaller figure, smaller font size for labels

%\FloatBarrier

\paragraph{Computational Overhead}
\label{par:compoverhead}

First, we investigate the computational overhead caused
by the continuous adaptation of all \params{}.
For this purpose, we measured the time \algref{alg:full}
spends on the adaptation and the set propagation
over 50 randomly-generated systems of dimensions 5 to 75.
\figref{fig:overhead} shows the fraction of the total time
spent on the parameter tuning:
The overhead remains manageable even for high-dimensional systems,
settling at about the same time as used for the set propagation.
This is achieved by the computational efficiency of \propref{prop:bloat}.
%in combination with the time-saving improvements mentioned at the start of this section.
The overhead is more than compensated by the adaptive adjustment of the \params{}
as discussed at the end of this section.
Also, the common practice of trial and error requires tuning times
that exceed computation times by several factors.

\paragraph{ARCH benchmarks}
\label{par:ARCH}

Next, we have applied our algorithm to the
48-dimensional Building (BLD) benchmark
and the 273-dimensional International Space Station (ISS) benchmark,
both from the ARCH competition \cite{ARCH19}.
These systems are manually tuned for the competition
by executing many runs to optimize the algorithm parameters
with respect to the computation time
while still satisfying all specifications.
For a fair comparison, we set $\bloatmax{}$
to the highest possible value that still verifies the given specification.
%This corresponds to setting the maximum error $\bloatmax{}$
%to the highest possible value that still verifies the given specification,
%thereby allowing a fair comparison.
We compare our results to the tool \textit{CORA} \cite{Althoff2015a}
since it is also implemented in \textit{MATLAB}.
%In \figref{fig:adapvsCORA}, the reachable sets for the benchmark ISSF01 are shown:
%The results are very similar since both aim to barely meet the specifications.

%\FloatBarrier
%\vspace{-0.5cm}

%\begin{figure}
%    \centering
%    \setlength{\belowcaptionskip}{-10pt}
%    \begin{subfigure}[b]{0.22\textwidth}
%        \includegraphics[width=\textwidth]{Figures/ISSF01_adap_300.pdf}
%%        \caption{}
%%        \label{fig:gull}
%    \end{subfigure}
%    ~ %add desired spacing between images, e. g. ~, \quad, \qquad, \hfill etc. 
%      %(or a blank line to force the subfigure onto a new line)
%    \begin{subfigure}[b]{0.22\textwidth}
%        \includegraphics[width=\textwidth]{Figures/ISSF01_CORA_300.pdf}
%%        \caption{A tiger}
%%        \label{fig:tiger}
%    \end{subfigure}
%    \caption{Benchmark ISSF01: $\Rti{[0,2]}$ of \algref{alg:full} (left), \textit{CORA} (right).}
%    \label{fig:adapvsCORA}
%\end{figure}

%\FloatBarrier

\begin{table*}[t]
\begin{center}
	\caption{Results of ARCH benchmarks for \algref{alg:full} and \textit{CORA}.}
	\label{tab:ARCH}
%	\begin{tabular}{l
%					c @{\hspace{8pt}}
%					c @{\hspace{8pt}}
%					c @{\hspace{8pt}}
%					c
%					c @{\hspace{8pt}}
%					c}
%		\toprule
%		\multirow{2}{3cm}{\textbf{Benchmark}} & \multicolumn{3}{c}{\textbf{\algref{alg:full}}} & &
%			\multicolumn{2}{c}{\textbf{CORA}} \\ \cmidrule{2-4} \cmidrule{6-7}
%		 & $\bloatmax{}$ & Time & Steps & & Time & Steps \\ \cmidrule{1-7}
%		BLDC01 & $2 \cdot 10^{-3}$ & 4.0s & 839 & & 5.5s & 2400 \\
%		BLDF01 & $6 \cdot 10^{-3}$ & 5.6s & 818 & & 6.0s & 2400 \\
%		ISSC01 & $5.6$ & 21.7s & 189 & & 22.8s & 1000 \\
%		ISSF01 & $2 \cdot 10^{-3}$ & 295s & 1216 & & 849s & 2000 \\
%		\bottomrule 
%	\end{tabular}
	\begin{tabular}{l
					c @{\hspace{7pt}}
					c @{\hspace{7pt}}
					c @{\hspace{7pt}}
					c @{\hspace{7pt}}
					c
					c @{\hspace{7pt}}
					c @{\hspace{7pt}}
					c}
		\toprule
		\multirow{2}{2.1cm}{\textbf{Benchmark}} & \multicolumn{4}{c}{\textbf{\algref{alg:full}}} & &
			\multicolumn{3}{c}{\textbf{CORA}} \\ \cmidrule{2-5} \cmidrule{7-9}
		 & $\bloatmax{}$ & Time & Steps & $[\Delta t_{\text{min}}, \Delta t_{\text{max}}]$ &
		 & Time & Steps & $[\Delta t_{\text{min}}, \Delta t_{\text{max}}]$ \\ \cmidrule{1-9}
		BLDC01 & $2 \cdot 10^{-3}$ & 4.0s & 839 & $[0.0081,0.0448]$ & & 5.5s & 2400 & $[0.0020,0.0100]$ \\
		BLDF01 & $6 \cdot 10^{-3}$ & 5.6s & 818 & $[0.0081,0.0597]$ & & 6.0s & 2400 & $[0.0020,0.0100]$ \\
		ISSC01 & $5.6$ & 21.7s & 189 & $[0.0704,0.1371]$ & & 22.8s & 1000 & $[0.0200,0.0200]$ \\
		ISSF01 & $2 \cdot 10^{-3}$ & 295s & 1216 & $[0.0059,0.0395]$ & & 849s & 2000 & $[0.0100,0.0100]$ \\
		\bottomrule 
	\end{tabular}
\end{center}
\vspace{-20pt} %20
\end{table*}

\tabref{tab:ARCH} shows a comparison between \algref{alg:full} and \textit{CORA}
in terms of the computation time, the number of steps,
and the minimum and maximum values for~$\Delta t$.
\algref{alg:full} adapts the values of the \params{} depending on the current system behavior
by enlarging the time step size in regions where this only moderately increases the over-approximation.
%which is exploited in converging or slow regions, where the time step size is increased
%and the set representation is simplified without losing precision
%since we are guaranteed to stay within the defined error bound.
%This behavior is similar for varying values of $\bloatmax{}$
%and can be observed in all four benchmarks
The resulting range can be observed by the large differences
between $\Delta t_\text{min}$ and $\Delta t_\text{max}$.
Consequently, the total number of steps decreases making less computations
necessary than in the case of fixed algorithm parameters as used by \textit{CORA}
and shown in \tabref{tab:ARCH}.
The range for $\Delta t$ in both building benchmarks by \textit{CORA}
is explained by the switching between two manually-tuned time steps.

%\FloatBarrier
%\vspace{-0.3cm}

%\FloatBarrier

%\FloatBarrier
%\vspace{-0.5cm}

\begin{figure}[h]
\begin{center}
	\setlength{\belowcaptionskip}{-25pt}
	\includegraphics[width = 0.48 \textwidth]{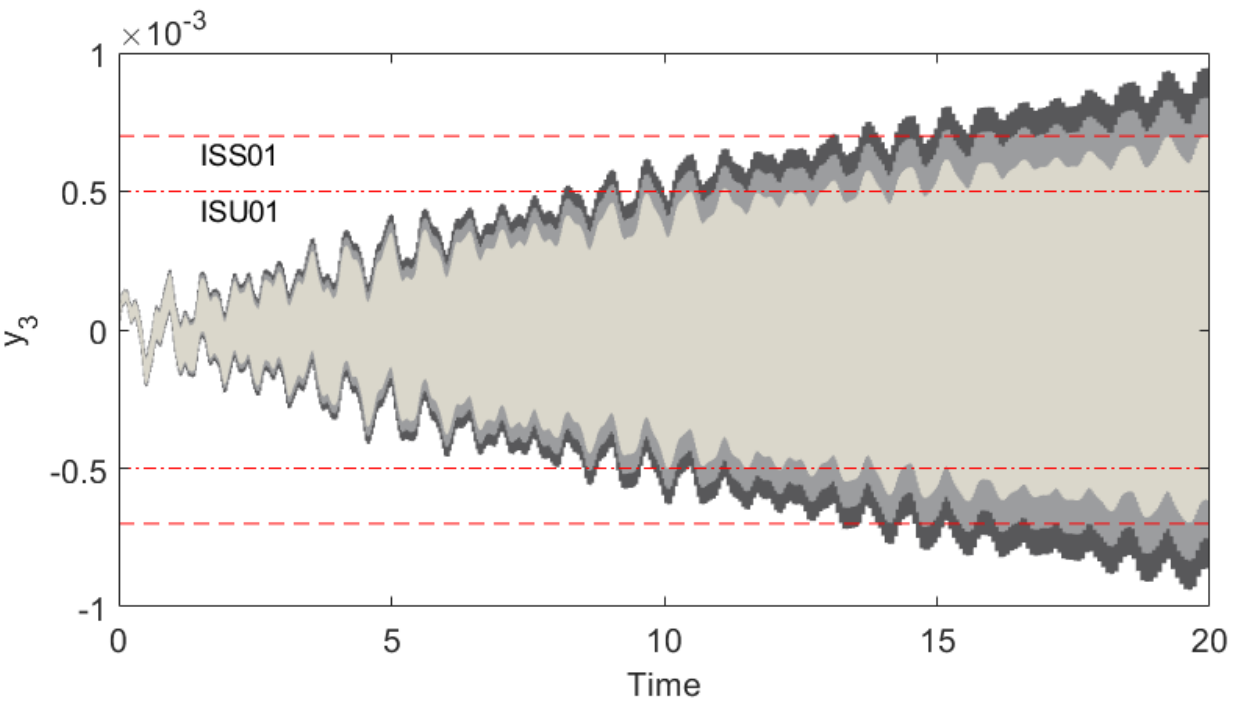}
	\caption{Benchmark ISSF01 with 
		the specifications ISS01, ISU01,
		and the reachable sets $\Rti{[0,T]}$ in
		dark gray ($\bloatmax{} = 20 \cdot 10^{-3}$),
		light gray ($\bloatmax{} = 10 \cdot 10^{-3}$),
		and ivory ($\bloatmax{} = 2 \cdot 10^{-3}$).}
	\label{fig:diffbloatmax}
\end{center}
\end{figure}

%\FloatBarrier
%\vspace{-1.4cm}

\begin{figure*}
    \centering
    \setlength{\belowcaptionskip}{-15pt}
    \begin{subfigure}[b]{0.3\textwidth}
        \includegraphics[width=\textwidth]{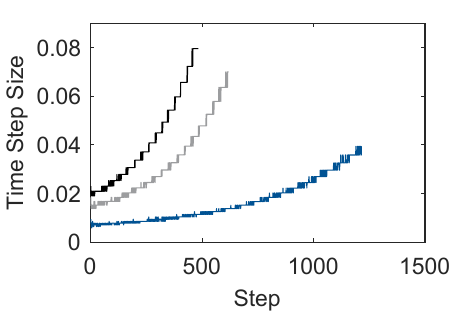}
%        \caption{}
%        \label{fig:gull}
    \end{subfigure}
    ~ %add desired spacing between images, e. g. ~, \quad, \qquad, \hfill etc. 
    \begin{subfigure}[b]{0.3\textwidth}
        \includegraphics[width=\textwidth]{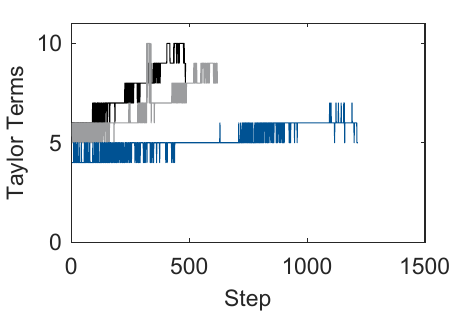}
%        \caption{A tiger}
%        \label{fig:tiger}
    \end{subfigure}
    \begin{subfigure}[b]{0.3\textwidth}
        \includegraphics[width=\textwidth]{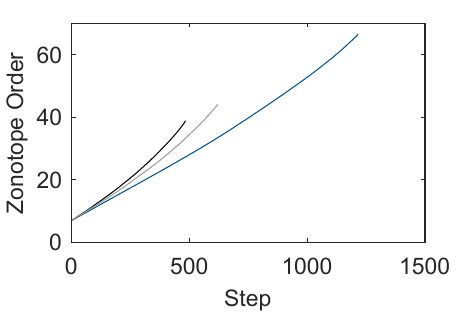}
%        \caption{A tiger}
%        \label{fig:tiger}
    \end{subfigure}
    \captionsetup{
    	width=\textwidth,  % width of caption is 100% of current textwidth
    	labelfont=bf,        % the label, e.g. figure 12, is bold
    	font=small,          % the whole caption text (label + content) is small
    	%format=hang,         % no caption text under the label
	}
	\caption{Benchmark ISSF01: $\Delta t, \eta, \zonOrder{}$ for different $\bloatmax{}$:
		black ($\bloatmax{}=20\cdot 10^{-3}$),
		gray ($\bloatmax{}=10\cdot 10^{-3}$),
		and blue ($\bloatmax{}=2\cdot 10^{-3}$).}
	\label{fig:ISSFparams}
% \end{center}
\end{figure*}

%\FloatBarrier
%\vspace{-1cm}

\figref{fig:diffbloatmax} shows the reachable sets for the benchmark ISSF01
using different values of $\bloatmax{}$:
The smaller the defined error bound, the tighter are the reachable sets.
We also recognize the linear increase of the admissible error bound over time
as the computed sets differ more towards the end of the time horizon.
\figref{fig:ISSFparams} shows the evolution of all \params{}
corresponding to the different values of $\bloatmax{}$:
%The choice of $\bloatmax{}$ does not have any effect
%on the adaptive properties of \algref{alg:full}
%since a similar behavior is observed for all three parameters.
A higher value for $\bloatmax{}$ yields a larger initial $\Delta t$
and a smaller total number of steps.
Note that the switching between previously-computed values of $\Delta t$ does not add
any computations since the sets are read from memory once they are computed.
%\textcolor{red}{The drop at the end is explained by the remainder
%of the time horizon being chosen in the last step.}
The number of Taylor terms $\eta$ is chosen jointly with $\Delta t$
and increases towards the end of the time horizon
to facilitate larger time step sizes.
We also observe that the zonotope order $\zonOrder{}$
reaches a higher maximum for smaller values of $\bloatmax{}$
since in that case we cannot reduce as much as for larger $\bloatmax{}$.

%\FloatBarrier

\paragraph{Genetic Algorithm Comparison}
\label{par:GA}

Finally, we want to compare our approach to %the solution obtained by 
a genetic algorithm searching for $\Delta t$, $\eta$, and~$\rho$.
To this end, we use the \emph{MATLAB} built-in genetic algorithm function.
While the parameters $\eta$ and $\rho$ are fixed,
we model the time step size by a polynomial up to order 2:
$\Delta t(t) = a + bt + ct^2$.
We restrict these parameters by the ranges
$\eta \in [1, 10]$, $\rho \in [2, 1000]$,
$a \in [0.0003,0.3], b \in [-0.1,0.1], c \in [-0.033,0.033]$.
The chosen bounds for $a,b$, and $c$ prevent $\Delta t$
from too drastic growth or shrinkage,
thereby focussing on suitable curves of $\Delta t$.
Higher orders did not provide any benefits.

In order to establish a level playing field,
we terminate once the obtained reachable set is within the box enclosure
of the reachable set of the adaptive algorithm enlarged by 10\%.
For computational efficiency interval over-approximations were used for this comparison.
The cost function is chosen as the maximum distance
to the enlarged adaptive reachable set over all dimensions.

The parameters specific to the genetic algorithm
have been set as follows:
We enable an infinite number of generations
with a maximum of 3 stall generations.
We aim to speed up the convergence by setting only 10 members per generation
as the evaluation of a single member is costly in higher dimensions.
%We thereby aim to speed up the convergence process.
For the members of the next generation, we use a standard crossover fraction of 0.75
and set the elite count to 1, % so that the best solution
%is carried over to the next generation.
carrying the best solution over to the next generation.

We applied \algref{alg:full} and the genetic algorithm
on 50 randomly-generated systems per dimension.
\tabref{tab:GA_systems} compares the average computation time
over varying dimensions of \algref{alg:full}
to the time the genetic algorithms takes until convergence.
The results show that \algref{alg:full} outspeeds all genetic algorithms.
Since \algref{alg:full} tunes the \params{} during runtime,
we only need a single iteration for the computation of the reachable set.
Contrary, the genetic algorithms run over many generations
repeatedly computing the reachable set while iteratively improving the solution
by means of the cost function.
This process is far more time-consuming than
the overhead caused by the adaptive parameter tuning.
The genetic algorithm using the constant polynomial for $\Delta t$
is faster than the higher-order polynomials as they re-compute
auxiliary reachable sets due to the non-constant time step size.
%
%\FloatBarrier
%\vspace{-0.7cm}

%\FloatBarrier
\vspace{-0.3cm}

\begin{table}[ht]
\begin{center}
	\caption{Computation time for \algref{alg:full} and the genetic algorithm (GA)
		averaged over 50 randomly-generated systems per dimension.}
	\label{tab:GA_systems}
	\begin{tabular}{l @{\hspace{20pt}}
	                c @{\hspace{8pt}}
	                c @{\hspace{8pt}}
	                c @{\hspace{8pt}}
	                c @{\hspace{8pt}}
	                c @{\hspace{8pt}}
	                c @{\hspace{8pt}}
	                c}
		\toprule
		& \multicolumn{7}{c}{\textbf{Dimension}} \\
		& 5 & 10 & 15 & 20 & 25 & 30 & 40 \\ \cmidrule{2-8}
		\algref{alg:full}            & 0.14s &  0.22s & 0.40s & 0.58s & 1.0s & 1.9s & 6.7s \\
		GA (order: 0) & 1.6s  &  3.4s  & 7.0s  & 10s   & 20s  & 30s  & 50s \\
		GA (order: 1) & 4.1s  &  9.4s  & 13s   & 21s   & 28s  & 40s  & 70s \\ 
		GA (order: 2) & 5.1s  &  10s   & 14s   & 21s   & 42s  & 58s  & 97s \\ 
		\bottomrule 
	\end{tabular}
\end{center}
\vspace{-10pt}
\end{table}

% server run:
%          5(50)   : 10(50)  : 15(50)  : 20(50)  : 25(50)  : 30(50)  : 40(50)  : 50(3)
% adap:    0.1407  |  0.2234 |  0.3947 |  0.5764 |  1.0224 |  1.9078 |  6.6806 | 12.0064  |
% const:   1.5966  |  3.4101 |  7.0105 | 10.3951 | 19.5972 | 30.3535 | 49.6641 | 77.5052  |
% linear:  4.0803  |  9.4287 | 13.4613 | 21.4503 | 28.3760 | 39.7459 | 69.5674 | 249.2785 |
% quad:    5.1416  | 10.0973 | 14.3158 | 21.1078 | 41.5667 | 57.6620 | 96.9833 | 146.5403 |

%\FloatBarrier

\paragraph{Discussion}
\label{par:discussion}

Our framework can also be applied to
other computations of reachable sets and other set representations.
In order to guarantee convergence and termination for all $\bloatmax{}~\in~\R{}$,
\thmref{thm:alg:full} has to hold for the applied error terms,
similarly as shown in \sectref{ssec:verass} for the presented implementation:
%For other equations computing the reachable sets,
%the error sets $\HRplusset{}$ and $\PRplusset{}$ containing over-approximative terms have to be specified
%in order to derive the necessary properties as shown in \propref{prop:bloatHRti_cora}
%for the homogeneous solution and \propref{prop:bloatPRti_cora}
%as well as \thmref{thm:bloatPRti_cora} for the inhomogeneous solution.
A tool developer has to modify
%two aspects:
%A different $\AppModPmax{}$ has to be fixed.
%In the case of, e.g., support functions or template polyhedra as the set representation,
$\AppModPmax{}$ and $\SetReprP{}$ which are, e.g., the number of template directions
when using template polyhedra.
%the parameters $\SetReprP{}$ are the number of template directions.
A corresponding error term for the set representation has to be defined.

%The definition of $\errOp{\cdot}$ in \eqref{eq:bloat}
%makes use of the interval over-approximation of the error set
%to compute the error in each iteration.
%Therefore, it is questionable whether the time step size is reduced too much
%if the difference between the true error and the computed error is large.
The choice of $\errOp{\cdot}$ in \eqref{eq:bloat}
does not add much over-approximation as can be observed from
comparing the ranges for the time step size $[\Delta t_\text{min}, \Delta t_\text{max}]$
to the manually-tuned $\Delta t$ by \textit{CORA} in \tabref{tab:ARCH},
where we see that \algref{alg:full} chooses a similar time step size,
yielding comparable results both in terms of the tightness
and the computational efficiency.
% with respect to the manually-tuned case.

%The \params{} are chosen according to the current dynamics,
%as opposed to fixed parameters which always have to be set conservatively
%in order to cope with the changing dynamics
%to avoid risking a coarse over-approximation at any time.
%Thus, the time spent on the propagation of the reachable sets
%is much larger in the case of fixed parameters than compared to
%the adaptive tuning as shown in \tabref{tab:ARCH}.
%This improvement compensates the tuning overhead in all of our experiments.
%In \tabref{tab:ARCH}, we observe that the adaptation of the time step size~$\Delta t$
%leads to larger values when appropriate, and consequently to a lower number of total steps.
%This compensates the computational overhead caused by the parameter tuning,
%as the adaptive tuning is computationally cheaper compared to the many additional propagations
%when using fixed values for the \params{}.
%This results in an overall lower computation time in all four benchmark evaluations.

The only remaining parameter for the practitioner to set is the error $\bloatmax{}$.
As shown in \figref{fig:diffbloatmax}, increasing the value of $\bloatmax{}$
results in a more over-approximative reachable set and vice versa.
Thus, the setting of $\bloatmax{}$ is intuitive
and can easily be adjusted for any system.
%Contrary, the setting of suitable values for the individual \params{}
%depends strongly on the investigated system
%and requires many iterations of trial and error.
%Our framework enables practitioners to use reachability analysis
%without requiring any knowledge about the \params{} used
%within the computation of the reachable sets.
%Due to the intuitive setting of the error bound,
%every practitioner can obtain a good solution.
%The adaptation during runtime provides
%the additional benefit of saving many propagation steps
%since the \params{} react to the changing dynamics
%which is impossible when using manually-tuned fixed parameters.

% -------------------------------------------------------
% CONCLUSION

\section{Conclusion}
\label{sec:conclusion}

In this paper, we presented a novel generic framework to automatically tune all \params{} which is a major problem of present reachability algorithms.
The presented algorithm enables a fully-automated computation of the reachable set whose error is below a user-defined error.
Previous work only considered the tuning of time parameters.
An example implementation has shown to outperform manually-tuned \params{} on benchmarks and provides better results than genetic algorithms searching for \params{} on randomly-generated systems of varying dimensions.
The extension of the presented framework to nonlinear systems will be considered in the future.

%
%%%%%%%%%%%%%%%%%%%%%%%%%%%%%%%%%%%%%%%%%%%%%%%%%%%%%%%%%%%%%%%%%%%%%%%%%%%%%%%%%
%\section{Acknowledgments}
%
%The authors gratefully acknowledge partial financial supports
%from the research training group CONVEY funded by the German Research Foundation under grant GRK 2428.
%The authors gratefully acknowledge the contribution of National Research Organization and reviewers' comments.
%
%
%%%%%%%%%%%%%%%%%%%%%%%%%%%%%%%%%%%%%%%%%%%%%%%%%%%%%%%%%%%%%%%%%%%%%%%%%%%%%%%%%

\bibliographystyle{ieeetr}
\bibliography{root}

\end{document}